\documentclass[12pt]{amsart}
\usepackage[utf8]{inputenc}
\usepackage{lmodern}

\usepackage{mathabx}
\usepackage{mathtools, amsmath, amsthm, amssymb, bm, centernot, cancel, scalerel}
\usepackage{enumerate}
\usepackage{stmaryrd, algorithm, algpseudocode}
\usepackage{enumitem, hyperref, caption, accents, soul, xcolor}
\usepackage{multirow, multicol, tabularx}
\usepackage{subfigure}
\usepackage{tikz, pgfplots, pgfplotstable}
\usetikzlibrary{arrows, arrows.meta}
\usetikzlibrary{angles, quotes}

    \usepackage[left=0.75in, right=0.75in, top=1in, bottom=1in]{geometry}
    \usepackage{enumitem}
    \usepackage{lipsum}
    \setlist[itemize]{leftmargin=*}

    \usepackage[
    backend=biber,
    style=alphabetic,
    sorting=nyt,
    giveninits=true
    ]{biblatex}
\AtEveryBibitem{
    \clearfield{urlyear}
    \clearfield{urlmonth}
    \clearfield{url}
    \clearfield{doi}
    \clearfield{issn}
    \clearfield{isbn}
    \clearfield{note}
    \clearfield{eprint}
}

    \hypersetup{
        colorlinks=true, 
        linkcolor=blue,
        citecolor=red,
    }

    \newtheorem{theorem}{Theorem}[section]  
    \newtheorem{definition}[theorem]{Definition}
    \newtheorem{corollary}[theorem]{Corollary}
    \newtheorem{lemma}[theorem]{Lemma}
    \newtheorem{remark}[theorem]{Remark}
    
    \newtheorem{proposition}[theorem]{Proposition}
    \newtheorem{assump}[theorem]{Assumption}

    \numberwithin{equation}{section}

    \let\pa\partial

    \newcommand{\R}{\mathbb R}
    \newcommand{\bA}{\mathbf A}
    \newcommand{\bB}{\mathbf B}
    \newcommand{\bC}{\mathbf C}
    
    \newcommand{\bE}{\mathbf E}
    \newcommand{\bF}{\mathbf F}

    \newcommand{\bI}{\mathbf I}
    \newcommand{\bJ}{\mathbf J}
    \newcommand{\bP}{\mathbf P}

    \newcommand{\bN}{\mathbf N}

    \newcommand{\bR}{\mathbf R}
    \newcommand{\bS}{\mathbf S}

    \newcommand{\ba}{\mathbf a}
    \newcommand{\bb}{\mathbf b}

    \newcommand{\bn}{\mathbf n}
    
    \newcommand{\be}{\mathbf e}
    
    \newcommand{\bq}{\mathbf q}
    \newcommand{\br}{\mathbf r}
    \newcommand{\bt}{\mathbf t}

    \newcommand{\bT}{\mathbf T}
    \newcommand{\bu}{\mathbf u}
    
    \newcommand{\bv}{\mathbf v}
    \newcommand{\bw}{\mathbf w}
    
    \newcommand{\bx}{\mathbf x}

    \newcommand{\bbf}{\mathbf f}

    \newcommand{\G}{\Gamma}

     \newcommand{\divM}{{\mathop{\,\rm div}}_{\!M}}
    
     \newcommand{\DivM}{{\mathop{\,\rm Div}}_{\!M}}
     \newcommand{\CurlG}{{\mathop{\,\rm Curl}}_{\Gamma}}
     \newcommand{\CurlM}{{\mathop{\rm Curl}}_{\!M}}
      \newcommand{\curlGscal}{{\mathop{\,\rm \mathbf{curl}}}_{\Gamma}}
     \newcommand{\CurlGvec}{{\mathop{\rm \mathbf{Curl}}}_{\Gamma}}
      
     \newcommand{\CurlMvec}{{\mathop{\,\rm \mathbf{Curl}}}_{\!M}}
     \newcommand{\curlG}{{\mathop{\,\rm curl}}_{\Gamma}}

    \newcommand{\nablaM}{\nabla_{\!\!M}}

    \newcommand{\odotcolon}{:}

    \renewcommand{\div}{\textrm{div}\ \!}

    \newcommand{\tr}{{\rm tr}}

    \DeclareGraphicsExtensions{.pdf,.eps,.ps,.eps.gz,.ps.gz,.eps.Y}

    \newcommand{\bsigma}{\boldsymbol{\sigma}}

    \usepackage{mathtools}

    \newcommand\wwidehat[1]{\arraycolsep=0pt\relax%
    \begin{array}{c}
    \stretchto{
      \scaleto{
        \scalerel*[\widthof{\ensuremath{#1}}]{\kern-.5pt\bigwedge\kern-.5pt}
        {\rule[-\textheight/2]{1ex}{\textheight}} 
      }{\textheight} %
    }{0.5ex}\\           
    #1\\                 
    \rule{*ex}{0ex}
    \end{array}
    }

    \def\XXint#1#2#3{{\setbox0=\hbox{$#1{#2#3}{\int}$ }
    \vcenter{\hbox{$#2#3$ }}\kern-.6\wd0}}

    \pgfplotsset{compat=1.18}

\usepackage[normalem]{ulem}

\title{Practical tensor calculus on embedded \\ submanifolds  of arbitrary codimension}
\author{Vladimir Yushutin\\ 
	vyushuti@utk.edu \\ Department of Mathematics, University of Tennessee, Knoxville}

\date{}

\begin{document}
\begin{abstract}
 We present a fully extrinsic, parametrization-free variant of tensor calculus on embedded, possibly evolving, submanifolds with boundary in arbitrary dimension and codimension. The proposed approach is component-free and, for general rank tensors, covers fundamental concepts such as  tangential projection,  extrinsic and covariant derivatives, the extrinsic Stokes' formula, and the Laplace-Beltrami operator. The distinctive features of the developed framework are its algorithmic recursivity and the transparency provided by the special row representation of tensors reminiscent of the recursive data structure of complete trees. Consequently, the suggested tensor calculus is amenable to computations and theoretical analysis, and the latter is demonstrated for general dimension and codimension through three standalone applications. First, we derive a new extrinsic conservation law, namely the principle of vanishing extrinsic momentum, for incompressible Euler flows on Riemannian manifolds. Second, we revisit the concept of Cauchy stress on embedded submanifolds with positive codimension and argue that the conservation of angular momentum implies tangentiality and symmetry of the stress tensors only when they are restricted to act on tangential orientations. Third, for evolving submanifolds, we introduce the material derivative of tensor fields of general rank in an extrinsic manner and derive an expression for the rate of change of the associated tensorial Dirichlet energy. The paper provides a practical notation and tools that are immediately usable in mathematical modeling, analysis of geometry-aware PDEs and in numerical methods on embedded submanifolds.
\end{abstract}

\maketitle

\textbf{Keywords}: 
\setcounter{tocdepth}{2} 

\section{Introduction}\label{sec:intro}
Many scientific problems involve submanifolds embedded in $\mathbb{R}^n$ and tensor fields on them. The related mathematical tools, such as differential operators, product and integration-by-parts rules, are often introduced intrinsically using local coordinate parametrizations and/or relying on Ricci index notation, see \cite{betounes1986kinematics, fosdick2009surface, jaric2020transport, NITSCHKE2022104428, nitschke2023tangential, nitschke2023tensorial, grinfeld2013introduction, pollard2025gauge}. Although powerful, this formalism can nevertheless hinder computations, especially in the presence of geometric evolution.  To this end, we propose a parametrization- and component-free version of tensor calculus that, as we believe, is extremely well-suited for computations on submanifolds of any dimension and codimension.

We  rely on the ambient Cartesian derivative in $\mathbb{R}^n$ and employ the extrinsic level set description of the submanifold's geometry \cite{osher1988fronts}. This central idea is not novel and has been tremendously successful in various applications, so the following papers represent only a selection of the most relevant works in the current context. In \cite{burman_cut_2018}, the authors considered a computational method for an elliptic PDE on a submanifold and for general dimension and codimension but only in the scalar case of the involved Laplace-Beltrami operator. Although general rank tensors were considered within the parabolic setting in \cite{bachini2024diffusion} on a two-dimensional surface embedded in $\mathbb{R}^3$, explicit expressions and computational results are presented only for tensor problems of rank at most two, see \cite[Section 2.2.2, Section 4]{bachini2024diffusion}. In \cite{fries2020unified}  equilibrium equations of the finite strain theory of ropes and membranes were presented and therefore the calculus there focuses on rank two tensors at most. Other notable applications that are restricted to the cases of codimension one  or  rank two tensor fields are \cite{hansbo2014finite, schollhammer2021higher, bouck_hydrodynamical_2024, jankuhn_incompressible_2018, olshanskii_finite_2018}.
Building upon and generalizing these references, we develop a notational framework which:
\begin{itemize}
    \item  considers general rank tensor fields on the submanifold,
    \item applies to arbitrary dimension and co‑dimension of the submanifold,
    \item  is algorithmically recursive,
    \item covers projection, gradient and the extrinsic Stokes' formula.
\end{itemize}
To demonstrate the practicality of the framework and show its main properties, we consider three complementing applications and derive novel results in Sections~\ref{sec:Euler_eq}, \ref{sec:stress_submanifold} and \ref{sec:evolving_dirichlet}. 

The first application in Section~\ref{sec:Euler_eq} concerns the well-known Euler equations on Riemannian manifolds \cite{arnold1998topological}. It is known, see, e.g., \cite{markakis2017conservation, gilbert2023geometric}, that any sufficiently smooth solution locally preserves the velocity component in the direction of the  Killing field for every symmetry of the manifold, just as translations of $\mathbb{R}^n$ correspond to the conservation of every component of the momentum vector.  Using the suggested framework for any embedding of the manifold, we discover a global  conservation principle  of the vanishing extrinsic momentum  even in the absence of any symmetries and for  general dimension and codimension. The principle can also be interpreted as an ambient force balance between the Young-Laplace force, the boundary reaction, and the centripetal force. The result generalizes an often overlooked property of \textit{zero net force exerted by pressure} on a fixed proper domain in $\mathbb{R}^{2}$ or $\mathbb{R}^{3}$  that carries the Euler flow \cite{constantin1988navier, Majda_Bertozzi_2001}.

In the second application in Section~\ref{sec:stress_submanifold} we revisit the concept of Cauchy stresses on embedded submanifolds. Recall that the Cauchy stress at a point is a rule that assigns a normalized force, called \textit{the stress vector}, for every infinitesimal oriented \textit{cut} of the material at that point. Historically, elasticity theory studying the  deformation of ropes, strings, beams, membranes, shells, and plates has been the major driver of the development of mathematical models posed on submanifolds.  In these cases of elastic structures of positive codimension, one faces the modeling choice of what constitutes an admissible cut for the concept of Cauchy stress -- think of an elastic, curved beam  and whether the loads on an oblique and on the perpendicular sections differ or not. Of course, such choice depends on the modeling goals, but most of the elastic structures above respond to an external load with a generally oriented stress vector field that rarely respects the tangent plane of the underlying submanifold. Yet, many other available in the literature continuum models on submanifolds of, typically, codimension one assume, sometimes implicitly, the tangentiality of the stress tensor, and references \cite{scriven1960dynamics, helfrich1973elastic, gurtin1975continuum} are often cited for a mathematical justification of this assumption; see, e.g., \cite{mohammadi2013surface,jankuhn_incompressible_2018, martinez2018temporal, pruss2021navier, olshanskii_finite_2021,chen2022gurtin, patil2024plane}. According to \cite{gurtin1975continuum}, as Newton laws and the principle of conservation of angular momentum lead to the symmetry of the stress tensor in classical continuum mechanics, they also imply its tangentiality for submanifolds of positive codimension. In our understanding, this implication is correct but only because the stresses in \cite{gurtin1975continuum} were assumed  to take only tangentially oriented cuts as input from the very start of the consideration -- think of a stress tensor on the aforementioned curved beam that disregards the possible oblique orientation of the cut and outputs the stress vector for the perpendicularly oriented component of the cut only. Although reasonable, the latter choice of admissible cuts requires further justification for every new model, especially for submanifolds of codimension greater than one. Motivated by the elaborated discussions in \cite{steigmann1999elastic,capovilla2002stresses,dharmavaram2025shear}, we will completely relax this assumption by allowing cuts that are generally oriented in $\mathbb{R}^n$, find the equations of force and torque equilibrium, which together imply the cancellation of the normal-at-tangential component of the stress vector, and argue that the Cauchy stress tensor on a submanifold may be non-symmetric and non-tangential  without violating   Newton laws of mechanics. Moreover, equilibrium equations for a general embedded submanifold  with arbitrary dimension and codimension will be discussed  unlike \cite{gurtin1975continuum}.

The third application in Section~\ref{sec:evolving_dirichlet}  studies an evolving submanifold and derives the temporal rate of change of the Dirichlet energy for a general rank tensor field. The Dirichlet energy involves the submanifold gradient and is a typical component of various continuum models and PDEs \cite{BarrettGarckeNuernberg2008, DziukElliott2010, ElliottStinner2010, grinfeld2013introduction, dziuk_finite_2013} on evolving domains. By introducing the material derivative extrinsically on an evolving submanifold, we compute generalizing \cite{dziuk_finite_2013} the rate of change of the tensorial Dirichlet energy for arbitrary dimension, codimension, and tensor rank.

The paper is organized as follows. Section~\ref{sec:prelim} introduces the
basics of the geometric  setting and further illustrates the case of low-dimensional submanifolds. The crucial recursive row representation of the general rank tensor is introduced in Section~\ref{sec:recursive} along with its tangential projection. In Section~\ref{sec:calculus} we develop a variant of tensor calculus on embedded submanifolds using the recursive representation of Section~\ref{sec:recursive}. Also, a discussion of the tensorial covariant derivative and the associated Laplace--Beltrami operator is provided. Sections~\ref{sec:Euler_eq}, \ref{sec:stress_submanifold} and \ref{sec:evolving_dirichlet} show the suggested tensor calculus in the three example applications discussed above for the arbitrary dimension and codimension. 
\section{Preliminaries}\label{sec:prelim}
\noindent
We use curly brackets $\{$ and $\}$ to form sequences (lists) of objects enclosed in them. For example,  $\{\be_x, \be_y, ..., \be_z\}$ and $\{\be^x, \be^y,  ..., \be^z\}$ denote the standard and dual ordered bases  of $\mathbb{R}^n$ indexed by $\mathcal{C}:=\{x,y,..., z\}$ and the $(n-3)$ omitted symbols will not be explicitly needed  in this paper. We use the asterisk symbol $*$ in  expressions to indicate that they are stated for every element of the index set $\mathcal{C}$. For instance, $\|\be_*\|_2=1$ means that every element of the standard basis has the unit Euclidean norm. An expression with $*$ that is premised with $\sum_*$ means the summation over all elements of $\mathcal{C}$. For example, while $\{f_*\}$ denotes the ordered tuple $\{f_x,f_y, ..., f_z\}$ of $n$ given functions, $\sum_* f_*$ means their sum $f_x+f_y+...+f_z$. 

We distinguish a vector $\bu$ and the corresponding linear functional denoted by $\bu^T$,
\begin{align*}
\bu:=u^x\be_x+u^y\be_y+...+u^z\be_z = \textstyle{\sum_*}u^*\be_*\,,\qquad\bu^T:=u_x\be^x+u_y\be^y+...+u_z\be^z=\textstyle{\sum_*}u_*\be^*\,,
\end{align*}
and  $u^*=u_*$ since the bases are orthonormal.  The corresponding duality trivially reads as follows:
\begin{align}\label{dual_pairing}
    \bu^T(\bv)=\bu\cdot\bv = u^x v^x+u^y v^y+...+u^z v^z =\textstyle{\sum_*} u^* v^*=\bv^T(\bu)\,.
\end{align}
In this paper a new notation is suggested that is convenient for the tensor calculus on submanifolds and is based on \eqref{dual_pairing} only. Because of this, the usage of coordinate components and the matrix notation is intentionally minimized. Nevertheless, the well-known definitions  \eqref{projectors} and \eqref{basics} are presented in the matrix notation, which is replaced by the row representation of tensors starting from Section~\ref{sec:calculus}. 

Consider an $(n-m)$-dimensional oriented submanifold  $M$  embedded in $\mathbb{R}^n$  given by the simultaneously vanishing level set functions $d_1,..., d_{m}\in{}C^2(\Omega_\delta)$,  $1 \leq m <n$. Here $\Omega_\delta:=\{\bx\in\mathbb{R}^n: \sum_{i=1}^m|d_i(\bx)|<\delta\}$ is a tubular neighborhood of $M$ of sufficiently small thickness $\delta>0$ and $m$ is the codimension of $M$. The unit vector field $\bn_i(\bx):=\nabla{}d_i/\|\nabla{}d_i\|\in{}C^1(\Omega_\delta)^n$ is orthogonal to the level sets of $d_i(\bx)$ for each $1\leq i\leq m$. The projection matrices $\bN,\, \bP\in C^1(\Omega_\delta)^{n\times n}$ onto the normal and tangent subspaces are given for every $\bx\in\Omega_\delta$ by
\begin{align}\label{projectors}
    \bN(\bx):=\bn_1\bn_1^T+...+\bn_m\bn_m^T\,,\qquad \qquad   \qquad \bP(\bx):=\bI-\bN(\bx)\,,
\end{align}
where $\bI$ is the identity matrix in $\mathbb{R}^n$ and $\bn_i\bn_i^T$, $1 \leq i \leq m$, denotes the column-row product. For general codimension $m$, the so-called \textit{submanifold gradients} $\nablaM f \in C(\Omega_\delta)^n$ of a scalar field $f\equiv f(\bx) \in C^1(\Omega_\delta)$ and $\nablaM \bu \in C(\Omega_\delta)^{n\times n}$ of a vector field $\bu\equiv \bu(\bx)\in C^1(\Omega_\delta)^n$ along with divergence $\divM\bu\in C(\Omega_\delta)$ are defined in 
\begin{align}\label{basics}
 \nablaM{}f :=(\nabla f)\bP\,,\qquad \qquad \nablaM\bu := (\nabla\bu)\bP
 \,,\qquad\qquad 
 \divM \bu  := \tr(\nablaM\bu)\,,
 \end{align}
 where the Cartesian gradient $\nabla{} f$ is a row and $\nabla\bu$ applies the scalar gradient to each component of the column $\bu$. Although the above definitions are given for fields in $\Omega_\delta$, the values of the operators therein for the points $\bx\in M$ only depend on the traces of their field arguments on $M$, see \cite{olshanskii_trace_2017} and Section~\ref{sec:material_perspect}. Therefore, \eqref{basics} is also applicable to fields that are given on $M$ only and this crucial aspect  has been useful in the formulation and the numerics of surface PDEs, i.e. in the case of codimension $m=1$, see \cite{bouck_hydrodynamical_2024}. Also,  submanifold gradients in \eqref{basics} admit the notion of weak derivatives for fields on $M$, see \cite{olshanskii_finite_2009}   for $m=1$ and the general case  follows similarly. For simplicity of presentation,  we assume that all tensor objects are sufficiently regular fields on $M$  and the operators in \eqref{basics} are well-defined for $M$. In light of this, we often omit $\bx \in M$ in the remainder of the paper.
 
 A general vector field $\bu$ on $M$ can be split into normal and tangential components,
 $\bu=\bP\bu+\bN\bu$.  The tangential vector field $\bP\bu$ can be studied via  local parametrizations to deduce intrinsic results that are readily transferable to the embedded $M$. For example, general Stokes' theorem on manifolds with boundary $\pa M$ implies for any tangential vector field $\bu=\bP\bu$ that
\begin{align}\label{generalStokes}
\int_M\divM (\bP\bu)
    &=\int_{\pa{}M}(\bP\bu)\cdot{}\bm{t}
    \,.\end{align}
    where a tangent to $M$ vector field $\bm t=\bP\bm t$ is the so-called outer \textit{co-normal} on the boundary $\pa{} M$. Note that the level set description of embedded submanifolds needs to be augmented to account for the boundary $\pa{}M$.  The integrals in  \eqref{generalStokes} are taken with respect to the induced measures.

    \subsection{Mean curvature vector}
Next, we introduce the mean curvature vector $\bm{\kappa}$ via the position vector  $\br(\bx)$ restricted to $\bx\in M$. Note that divergence of a matrix is applied row-wise.
\begin{definition}\label{mean_curvature_vector}
    The mean curvature vector $\bm\kappa\in C(\Omega_\delta)^n$ for any codimension $m$ is defined by 
\begin{align*}
\bm{\kappa}:=-\divM(\nablaM\br),\qquad{}\qquad{}\qquad{}\br(\bx):=\bx=(x,y,..., z)\,.
\end{align*}
\end{definition}
Because the operators in \eqref{basics} are row-wise, the components of $\bm\kappa$ can be expressed as in
\begin{align}\label{mean_curvature_vector_components}
    \bm{\kappa}_j=-\divM(\bP\nabla{x_j})  = -\divM(\bP\be_j)= -\divM \bP_j = \divM \bN_j\,,\qquad   1\leq \forall j \leq n\,,
\end{align}
where $\bP_j:=\bP\be_j$ and $\bN_j:=\bN\be_j$ denote the $j$-th columns of the matrices in \eqref{projectors}. Note that, for codimension $m=1$, the mean curvature vector  $\bm\kappa=\kappa\bn$  with the normal $\bn$ and the scalar mean curvature $\kappa = \tr(\nablaM\bn)$. For the general codimension $m>1$, $\bm\kappa$ is simply normal, i.e. $\bN\bm\kappa = \bm\kappa$, which is proven in the next proposition along with a generalization of \eqref{generalStokes} for non-tangential vector fields.
\begin{proposition}[extrinsic Stokes' formula]\label{vector_parts}
The mean curvature vector $\bm\kappa$ satisfies $\bP\bm\kappa=0$, and for any vector field $\bu$ on $M$ with boundary $\pa{}M$, we have
\begin{align*}
\int_M\divM\bu
    &=\int_{\pa{}M}\bP\bu\cdot{}\bm{t}+ \int_M \bN\bu\cdot \bm{\kappa}
    \,.\end{align*}
\end{proposition}

\begin{proof} Thanks to \eqref{generalStokes}, it remains to prove from \eqref{basics} that $\divM\bN\bu=\bu\cdot \bm{\kappa}$, which also yields $\bP\bm\kappa=0$ with $\bu\equiv \bP\bm\kappa$. Indeed, recalling the  product rules of the Cartesian derivative and the cyclic property of the matrix trace, one derives
    \begin{align*}
\divM(\bN\bu) &=\tr{(\nabla{(\bN\bu)\bP})}=\tr{(\bN(\nabla\bu)\bP)}+\tr{(\textstyle{\sum_j}\bu_j(\nabla\bN_j)\bP)}=\textstyle{\sum_i}\bu_j\tr{((\nabla\bN_j)\bP)}
\end{align*}
 for any $\bx\in\Omega_{\delta}$. Since $\tr{((\nabla\bN_j)\bP)}=\divM \bN_j$, for any $1\leq{} j \leq n$, the last term is the dot product of $\bu$ with  a vector whose components were given in \eqref{mean_curvature_vector_components}.
\end{proof}
\begin{corollary} Setting $\bu=f \be_j$ in Proposition~\ref{vector_parts} for every $1\leq j \leq n$ shows that
    \begin{align*}
        \int_M \nablaM f= \int_{\pa{}M} f \bt^T + \int_{M} f \bm{\kappa}^T\,,
        \end{align*} 
 which then yields outer products in
        \begin{align*}
        \int_M \nablaM \bu= \int_{\pa{}M} \bu  \bt^T + \int_{M} f\bu \bm{\kappa}^T\,.
    \end{align*}

\end{corollary}
\subsection{Fundamental theorem of calculus for one-dimensional paths}\label{sec:FTC}
In the case $n-m=1$ we denote $M\equiv \gamma$ and the projector can be expressed as $\bP=\bw\bw^T$ for a unit tangent vector field $\bw$ on $\gamma$. The boundary $\pa\gamma$  is zero-dimensional and consists of  the starting point $\ba\in\mathbb{R}^n$ with $\bw=-\bt$ and the endpoint $\bb\in\mathbb{R}^n$ with $\bw=\bt$. The operators \eqref{basics} are now expressed as follows,
\begin{align*}
  \nablaM{}f \equiv \nabla_{\!\!\gamma} f = (\nabla f)(\bw\bw^T) = 
((\nablaM f)\bw)\bw^T =(\mathcal{D}_{\bw}^\gamma f) \bw^T\,,
\\
    \divM\bu \equiv \div_{\!\!\gamma} \bu = \tr((\nabla\bu)(\bw\bw^T)) = 
\bw^T(\nabla\bu)\bw = \bw\cdot((\nabla\bu)\bw) = (\mathcal{D}_{\bw}^\gamma \bu) \cdot\bw \,,
\end{align*}
with the directional derivatives of a scalar $\mathcal{D}_{\bw}^\gamma f := (\nabla_\gamma f)\bw$ and a vector $\mathcal{D}_{\bw}^\gamma \bu:=(\nabla_\gamma \bu)\bw$ that will be generalized later in Definition~\ref{def:cartesian} to tensor fields of arbitrary rank. The extrinsic Stokes' formula \eqref{vector_parts} yields
pointwise evaluations at endpoints of $\gamma$ in the following,
\begin{align}\label{curve_vector}
\int_\gamma \mathcal{D}_{\bw}^\gamma \bu \cdot \bw =  (\bP\bu \cdot \bw)|_{\ba}^{\bb}  +\int_\gamma \bN\bu\cdot\bm{\kappa}\,,\qquad f|_{\ba}^{\bb}:= f[\bb] - f[\ba] \,.
\end{align}
Hence, the left-hand side must vanish for any tangential $\bu$ on a path $\gamma$ without boundary. Moreover, for any scalar function $f$ one can consider 
$\bu=f\bw=\bP\bu$ in \eqref{curve_vector} with $\nabla_\gamma \bu = \bw(\nabla_\gamma f) +f(\nabla_\gamma{}\bw)$ so  $\mathcal{D}_{\bw}^\gamma \bu \cdot \bw = \mathcal{D}_{\bw}^\gamma f+f\bw^T(\nabla_\gamma{}\bw)\bw$ and, consequently,
\begin{align}\label{FTC}
\int_\gamma \mathcal{D}_{\bw}^\gamma f =  f|_\ba^\bb\,,\qquad\qquad  \bw\cdot\bw=1\,,
\end{align}
since $0=\nabla_\gamma(\bw\cdot\bw)=2\bw^T(\nabla_\gamma \bw)$.

\subsection{Curl and circulation for two-dimensional surfaces}

Due to its paramount practicality, we will consider the circulation and the curl-curl structure for two-dimensional surfaces $M\equiv\Gamma$, i.e. $n-m=2$. In this case the boundary $\pa{}M\equiv\pa\Gamma$, if present, is one-dimensional and can be oriented by a vector field $\bm{\tau}=\bP\bm{\tau}$, $\bm{\tau}\cdot\bm{t}=0$, such that $\{\bm{t},\bm{\tau},\bn_1,...,\bn_m\}$ is a positively oriented  basis in $\mathbb{R}^n$.
    \begin{definition}\label{def:vector_scalar_curls}  For $n-m=2$ we define $\bu^\dagger:=0$ if  the tangential component $\bP\bu=0$, otherwise $\bu^\dagger$ is the vector $\bP\bu$ rotated in the tangent plane by $\pi/2$ so the basis $\{\bu, \bu^\dagger, \bn_1, ..., \bn_m\}$ is positively oriented in $\mathbb{R}^n$. We introduce the following operators
    \begin{align*}
 \curlG \bu := -\divM \bu^\dagger
\,,\qquad{}
\curlGscal{f} :=((\nablaM f)^T)^\dagger  \,,\quad 
 \end{align*}
    \end{definition}
Note that $\bu\cdot\bv = \bu^\dagger{}\cdot\bv^\dagger$ and $(\bu^\dagger)^\dagger=-\bu$ for a tangential $\bu=\bP\bu$ and any $\bv$.
\begin{remark} Consider $n=3$ and the plane $z=0$ with $\bn=\be_z$. For $\bu=(-y)\be_x+x\be_y$ we have $\bu^\dagger=\bn\times\bu$ and $ \curlG \bu=-\divM(\be_z\times((-y)\be_x+x\be_y))=-\divM((-y)\be_y+x(-\be_x))=2$.
\end{remark}
\begin{proposition}[circulation and curl-curl identity]\label{curl_int_parts}
For a vector field $\bu$ and a scalar field $f$ on a two-dimensional submanifold $\Gamma$ with boundary $\pa \Gamma$, we have
\begin{align*}
\curlG (f\bu) =f  \curlG \bu + \bu  \cdot \curlGscal{f} 
\,,\qquad \qquad \qquad \qquad
\int_\G\curlG \bu = \int_{\pa{}\G}\bu\cdot{}\bm{\tau}\,.
\end{align*}
\end{proposition}
\begin{proof} Proposition~\ref{vector_parts} for a two-dimensional  $M\equiv\Gamma$ yields with Definition~\ref{def:vector_scalar_curls} that
\begin{align}
\int_\G\curlG \bu=-\int_\G\divM \bu^\dagger
    &=-\int_{\pa{}\G}\bu^\dagger\cdot{}\bm{t}
 =\int_{\pa{}\G}\bu\cdot{}\bm{\tau}\,,
 \end{align}
  since both $\bm{\tau}$ and  $\bu^\dagger$ can be obtained from $\bm{t}$ and $\bu$ by rotating them by $\pi/2$ . Similarly, 
\begin{align}\label{curl_product}
  \curlG (f\bu) = -\divM(f\bu^\dagger)=-f \divM \bu^\dagger + \bu^\dagger\cdot\nablaM{}f
        =f  \curlG \bu + \bu  \cdot \curlGscal{f}   
\end{align}
 by the product identities from \cite{bouck_hydrodynamical_2024} and the properties of $\bu^\dagger$ from Definition~\ref{def:vector_scalar_curls}.
\end{proof}

\begin{corollary} For any scalar field $f$ on a two-dimensional submanifold $\Gamma$, we have\label{cor:vanishing_div_curl}
    \begin{align}
        \curlG (\nablaM f)^T = 0\,,\qquad\qquad\qquad\qquad \divM(\curlGscal{}f) = 0\,.
    \end{align}
\end{corollary}
\begin{proof}
    Consider any arbitrarily small submanifold $\tilde{\Gamma}\subset\Gamma$ and recognize its boundary as a one-dimensional path,  $\gamma\equiv\pa\tilde{\Gamma}$, equipped with the unit tangent field $\bw\equiv\bm\tau$, see Section~\ref{sec:FTC}. The integral of Proposition~\ref{curl_int_parts} with $\bu\equiv (\nablaM f)^T$  and \eqref{FTC} show
    \begin{align}
        \int_{\tilde{\Gamma}} \curlG (\nablaM f)^T=\int_\gamma (\nablaM f)^T\cdot \bm{\tau} =  \int_\gamma \mathcal{D}_{\bm\tau}^\gamma f = 0\,,
    \end{align}
    and the first claim follows as $\tilde{\Gamma}$ is arbitrary and $\nabla_\gamma f = \bm\tau^T (\nablaM{}f)\bm{\tau}$ as $\gamma \subset \Gamma$. The remaining part follows straight from Definition~\ref{def:vector_scalar_curls}, $ \divM(\curlGscal{}f)=-\curlG (\nablaM f)^T=0$.
\end{proof}

\section{Recursive representation of tensors of general rank}\label{sec:recursive} 

We develop a practical framework for tensor fields $\bT$ of rank $(0,q)$ on $M$ and the  identification of vectors and covectors is the only use of the superscript $\cdot^T$ in this paper. Since we use the standard Cartesian basis, a multilinear map $F$ of rank $(p,q)$ can be identified with a tensor $\bT$,
\begin{align*}
F(\bu^T_1,\bu^T_2,., \bu^T_p, \bv_1,\bv_2,., \bv_q) \equiv \bT(\bu_1,\bu_2,., \bu_p, \bv_1,\bv_2,., \bv_q)\,,
\end{align*}
and any formula with $\bT$ can be reinterpreted as a formula for the corresponding $F$. Note that the matrix notation utilized in Section~\ref{sec:prelim} will not be used hereafter. 
\subsection{Recursive row representation}
We focus on multilinear maps  of type $(0,q)$ determined by all $n^q$ outputs $\bT(\bv_1,.,\bv_q)\in\mathbb{R}$ for  arbitrary arguments $\bv_1,.,\bv_q \in \mathbb{R}^n$. The definition below introduces a  representation of such maps that  is recursive in $q$ thus conducive to computations  involving embedded manifolds, as we show later, and is solely based on the notion of Cartesian dual pairing \eqref{dual_pairing}. The direct analogy is how the matrix notation relies on the table of coefficients and the row-column rule for the evaluation of a given linear map on a vector. The following definition extends this practicality to arbitrary tensor fields of type $(0,q)$ and, in principle, does not have a direct relation to the embedded submanifold $M$.
\begin{definition}[row-representation]\label{def:row-represent}
Denote scalars by $\mathbb{T}_0$. Let $\mathbb{T}_q\equiv \mathbb{T}_q^n$, $q\geq 1$, be the set of maps $\bT:(\bv_1,.,\bv_q)\rightarrow \mathbb{R}$, $\bv_k \in \mathbb{R}^n$, $1\leq k\leq q$, defined recursively by
\begin{align*}
   \bT\equiv \{\bT^*\}&\equiv\{\bT^x,\bT^y,..., \bT^z\}\,, & \qquad & \bu^T\equiv \{\bu^*\}\equiv \{\bu^x, \bu^y,...,\bu^z\}\,,
   \\
   \bT(\bv_1,.,\bv_q)&:=\{\bT^*(\bv_2,.,\bv_{q})\}(\bv_1)\,, & \qquad & \bu^T(\bv)\equiv \{\bu^*\}(\bv)=\bu\cdot\bv\,,
\end{align*}
where $\bT^x ,\bT^y,...,\bT^z\in \mathbb{T}^n_{q-1}$ are called the \textit{row-components} of $\bT$. 
\end{definition}
\begin{remark}\label{tree_analogy}
    One can think of a complete $n$-ary tree of depth $q$: the $n^q$ values of the particular map  $\bT\in\mathbb{T}_q$ on all possible combinations of  $\be_*$ are stored at the leaf nodes, and the tree structure of recursive $\bT^*$ guides the evaluation of various tensor operations via   recursively applied dot products. 
\end{remark}
Due to the linearity of $\bu\cdot\bv$ and the recursive definition, any map $\bT\in\mathbb{T}_q$ is linear,
\begin{align*}
    \bT(\bv_1,.,f\bv_k+g\bu,.,\bv_q)=f\bT(\bv_1,.,\bv_k,.,\bv_q)+g\bT(\bv_1,.,\bu,.,\bv_q)\,,\qquad \forall f, g \in \R, \qquad \forall \bu\in\mathbb{R}^n\,, 
\end{align*}
 in every argument $\bv_k$, $1\leq k\leq q$, implying for the row components of a $\bT$ that 
\begin{align}\label{star_component}\bT^*(\bv_2,.,\bv_q)=\bT(\be_*,\bv_2,.,\bv_q)\,.
\end{align}The set $\mathbb{T}_q$ is an inner product space as the next two definitions imply.

\begin{definition}[linear structure]\label{def:linear_structure}
The following operations on $\mathbb{T}_q$ are defined recursively,
\begin{align*}
    f\bT &:= \{f\bT^*\}\,,\qquad & f\bu^T &:= (f\bu)^T,
    \\
    \bT+\bS &:= \{\bT^*+\bS^*\} ,\qquad & \bu^T+\bv^T &:= (\bu+\bv)^T\,.
\end{align*}
\end{definition}

\begin{definition} \label{def:inner_product} Frobenius inner product  $\bS\odot\bT\in\mathbb{T}_{0}$ of $\bT,\bS\in\mathbb{T}_{q}$ is defined recursively by
    \begin{align*}
    \bS\odot{}\bT &:=\sum_*\bS^*\odot{}\bT^*\,,\qquad \qquad  \quad   \qquad  \quad \bu^T\odot{}\bv^T:=\bu\cdot\bv \,.
 \end{align*}
\end{definition}
\begin{remark}
    The row representation of Definition~\ref{def:row-represent} is closely connected with the decomposition into the tensor products of the  basis linear functionals, i.e. into the sum of terms like $\be^x \otimes \be^y \otimes... \otimes\be^z$ scaled with the then called \textit{components} of the tensor. Collecting terms as in $\be^x \otimes (...)$, $\be^y \otimes (...)$ ,..., $\be^z \otimes (...)$ is equivalent to finding the row components in $\bT=\{\bT^*\}$ and vice versa.
\end{remark}

\subsection{Recursive tensor operations}

We now introduce well-known tensor operations along with their row representation from Definition~\ref{def:row-represent}.
\begin{definition}\label{def:outer_product} The outer product  $\bT\otimes\bS \in \mathbb{T}_{q+s}$   of a $\bT\in \mathbb{T}_q$ with a $\bS\in \mathbb{T}_s$ is defined  by
    \begin{align*}
        (\bS\otimes\bT)(\bv_1,.,\bv_s,\bv_{s+1}, .,\bv_{s+q})=\bS(\bv_1,.,\bv_s) \bT(\bv_{s+1},.,\bv_{s+q})\,.
    \end{align*}
\end{definition}
\begin{proposition}\label{outer_represent} The row-representation of the outer products is given recursively by
    \begin{align*}
      \bS \otimes \bT&:=\{\bS^*\otimes\bT\} \,,\qquad  \bu^T \otimes \bT:=\{\bu^*\bT\} \,.
     \end{align*}
\end{proposition}
\begin{proof}Direct application of recursive Definitions~\ref{def:row-represent} and \ref{def:outer_product}.
\end{proof}

\begin{definition}\label{def:left_right_insertions}
The left-insertion $\bT(\bv)\in \mathbb{T}_{q-1}$ and the right-insertion  $\bT\cdot \bv\in \mathbb{T}_{q-1}$ are defined in
 \begin{align*}
    \bT(\bv)(\bv_2,.,\bv_{q})&:=\bT(\bv,\bv_2,.,\bv_{q})\,, & \qquad  & \forall \bv_k \in \mathbb{R}^n\,,\quad  2\leq k\leq q\,,
    \\
    (\bT\cdot \bv)(\bv_1,.,\bv_{q-1})&:=\bT(\bv_1,.,\bv_{q-1},\bv)\,, & \qquad & \forall \bv_k \in \mathbb{R}^n\,,\quad  1\leq k\leq q-1\,,
\end{align*}
 for a $\bT\in\mathbb{T}_q$ and a given $\bv\in \mathbb{R}^n$.
\end{definition}
The next statement   can be interpreted  using the $n$-ary tree analogy: the left insertion is a linear combination of the $n$ root's children's subtrees with components of a fixed $\bv\in \mathbb{R}^n$, while the right insertion contracts every bottom, $\mathbb{T}_1$, leaf with $\bv$.
\begin{proposition}\label{insertion_rowrep}
    The row representations of   left- and right-insertions  are given in
    \begin{align*}
    \bT(\bv) &=\textstyle{\sum_*}\bv^*\bT^*\,,
    \\
   \bT\cdot \bv &= \{\bT^*\cdot \bv\}\,,\qquad \bu^T\cdot\bv:=\bu^T(\bv)\,,
\end{align*}
and the insertions commute, $\bT(\bu)\cdot\bv  = (\bT\cdot\bv)(\bu)$.
Moreover, $\bT\cdot\bv=0$ or $\bT(\bv)=0$ holds for all $\bv$ if and only if $\bT=0$. 
\end{proposition}
\begin{proof}
     Definitions~\ref{def:row-represent} and \ref{def:left_right_insertions} yield the first part of the claim and commutativity,
     \begin{align}
    \bT(\bu)\cdot\bv =  (\textstyle{\sum_*}\bu^*\bT^*)\cdot\bv = \textstyle{\sum_*}\bu^*(\bT^*\cdot\bv)=\textstyle{\sum_*}\bu^*(\bT\cdot\bv)^*
    = (\bT\cdot\bv)(\bu)\,.
\end{align}
Taking $\bv=\be_*$ shows $0=\bT(\be_*)=\bT^*$ and $\bT=0$. Similarly, $0=\bT\cdot\bv=\{\bT^*\cdot \bv\}$ implies $\bT^*\cdot \bv=0$, for all $\bv$,  which recursively yields $\bu^T\cdot\bv=0$, for all $\bv$, for every row component $\bu^T\in \mathbb{T}_1$ of $\bT$.   Thus, all $\bu^T=0$ and,  consequently, $\bT=0$.
\end{proof}
\begin{remark}
    Any rank 2 tensor  $\bT$  can be interpreted as a linear transformation $A:\mathbb{R}^n\rightarrow \mathbb{R}^n$ such that $\bu\cdot A(\bv)=\bT(\bu,\bv)$, for all $\bu,\bv\in \mathbb{R}^n$, with the standard matrix $A_{ij}$, $1\leq i,j\leq n$. The row component $\bT^*$ are the rows of $A_{ij}$, the right-insertion  $\bT\cdot\bv$ corresponds to the matrix-column product with $A_{ij}$ while the left-insertion $\bT(\bu)$  corresponds to the row-matrix product with $A_{ij}$  justifying their names.
\end{remark}

Next we interpret the left (right) insertion with a vector $\bu$ as the left (right) contractions with the covector $\bu^T \in \mathbb{T}_1$, i.e. 
\begin{align}\label{vector-covector_insert}(\bu^T :\bT) \,:=\,\bT(\bu)\,,\quad (\bT:\bu^T) \,:=\,\bT\cdot\bu\,,
\end{align}
and generalize them to contractions with general tensors of equal or lower rank.
 \begin{definition}\label{def:left_right_contraction} Left-contraction  $\bS:\bT\in\mathbb{T}_{q-s}$ and right-contraction $\bT\odotcolon\bS\in\mathbb{T}_{q-s}$ of $\bT\in\mathbb{T}_{q}$   and $\bS\in\mathbb{T}_{s}$, $q\geq s$,  are defined recursively by
    \begin{align*}
     \bS:\bT \,&:=\, \textstyle{\sum_*}\bS^*:\bT^*\,,\qquad  f:\bT=f\bT\,,
     \\
     \bT\odotcolon \bS\,&:=\,\{\bT^*\odotcolon\bS\}\quad\mathrm{for}\quad q>s,\qquad  {\bT}\odotcolon\bS\,:=\,{\bT}\odot{}\bS\quad\mathrm{for}\quad q=s\,.
 \end{align*}
\end{definition}
Note that if $q=s$ then the left and the right contractions coincide with $\bS\odot\bT=\bT\odot\bS$ but if the ranks are different then only one order of arguments is well-defined allowing to overload the notation and use it for both contractions. 
\begin{proposition}\label{associativity} For $\bS\in\mathbb{T}_{s}$, $\bT\in\mathbb{T}_{q}$, $q>s$, we have for any $\bv \in\mathbb{R}^n$ that
    \begin{align}
        (\bS:\bT)\cdot \bv = \bS:(\bT\cdot \bv)\,.
    \end{align}
\end{proposition}
\begin{proof} Proposition~\ref{insertion_rowrep} and \eqref{vector-covector_insert} provide with the base of induction over $s$ for any fixed $q$. Also,
     \begin{align*}
        (\bS:\bT)\cdot\bv =\sum_* (\bS^*:\bT^*)\cdot\bv=\sum_*\bS^{*}:(\bT^*\cdot \bv)=\sum_*\bS^{*}:(\bT\cdot \bv)^* = \bS:(\bT\cdot\bv)\,,
    \end{align*}
where we utilized the induction step in the second equality followed by Proposition~\ref{insertion_rowrep}.
    \end{proof}

\subsection{Tangent tensors and  the tangential projection }

Consider a fixed point $\bx\in M$ with projections   \eqref{projectors}  that can now be represented via Definition~\ref{def:row-represent} and  \eqref{mean_curvature_vector_components} as follows,
\begin{align}\label{projection-row-represent}
    \bI=\{\be^*\}\,,\qquad \bN=\{\bN_*^T\}=\sum_{i=1}^m\bn_i^T\otimes\bn_i^T\,,\qquad \bP=\{\bP_*^T\}=\bI-\bN=\{\be^*-\sum_{i=1}^m\bn_i^*\bn_i^T\}\,.
\end{align}
Similarly to the projection $\bP\bu$ of a vector at $\bx\in{}M$ onto the tangent plane, we can define pointwise the dual projection  $\mathbb{P}:\mathbb{T}_1\rightarrow \mathbb{T}_1$ of linear functionals $\bu^T$ by  its action $(\mathbb{P}\bu^T)(\bv):=\bu^T(\bP\bv)$ on every $\bv\in{}\mathbb{R}^n$. The  functionals satisfying $\mathbb{P}\bu^T=\bu^T$ form the  subspace $\mathbb{P}\mathbb{T}_1\subset \mathbb{T}_1$ that we call \textit{tangent}. We now generalize these concepts to tensors of arbitrary rank.
\begin{definition}\label{def:tangency}
   Let $\mathbb{PT}_0:=\mathbb{T}_0$. The tangent subspace  $\mathbb{PT}_q\subset \mathbb{T}_q$  and the tangential projection $\mathbb{P}:\mathbb{T}_q\rightarrow \mathbb{PT}_q$, $q\geq 1$, are defined as follows
\begin{align*}
    &\mathbb{PT}_q=\{\bT\in\mathbb{T}_q:\quad \bT(\bv_1,., \bN\bv_k,.,\bv_q)=0\,, \quad   1\leq \forall k\leq q\,,\,\,\forall \bv_1,.,\bv_q\in\mathbb{R}^n\}\,,
    \\
    &(\mathbb{P}\bT)(\bv_1,.,\bv_q):=\bT(\bP\bv_1,.,\bP\bv_q)\,,\qquad \forall \bv_1,.,\bv_q\in\mathbb{R}^n\,.
\end{align*}
\end{definition}
Note that the defining property of $\mathbb{PT}_q$  is equivalent to $\bT(\bv_1,., \bv_k,.,\bv_q)=\bT(\bv_1,., \bP\bv_k,.,\bv_q)$.
\begin{lemma}\label{lemma:t_n}
    If $\bT(\bv_1,., \bv_k,.,\bv_q)=\bT(\bv_1,., \bP\bv_k,.,\bv_q)$ for all $\bv_1,.,\bv_q\in\mathbb{R}^n$ for all $k$ except for $k=1$ then 
    \begin{align}\label{t_n}
       \bT_n:=\bT-\sum_{i=1}^m\bn_i^T\otimes\bT(\bn_i) \,\, \in \,\, \mathbb{PT}_q\,.
    \end{align}
\end{lemma}
\begin{proof} Under the condition of the statement, we only need to show for $\bT_n$ that 
\begin{align}
    \bT_n(\bP\bv_1, \bv_2,.,\bv_q)=\bT_n(\bv_1,\bv_2,.,\bv_q)\,,\quad \forall \bv_1, ..., \bv_q\in\mathbb{R}^n\,,
\end{align}
as $\bT-\bT_n$ along with $\bT$ satisfies Definition~\ref{def:tangency}  for every $2\leq k \leq q$ since for all $1\leq i \leq m$,
 \begin{align*}         (\bn_i^T\otimes\bT(\bn_i))(\bv_1,.,\bv_k,.,\bv_q)= (\bn_i^T(\bv_1))\bT(\bn_i,\bv_2,.,\bv_k,.,\bv_q)=(\bn^T_i\otimes\bT(\bn_i))(\bv_1,.,\bP\bv_k,.,\bv_q)\,,
    \end{align*}
    where we recall Definition~\ref{def:left_right_insertions} and Proposition~\ref{outer_represent}. Similarly, by construction of $\bT_n$, we compute
    \begin{align}\label{t_n=t}
    \bT_n(\bv_1)=\bT(\bv_1)-\sum_{i=1}^m(\bn^T_i(\bv_1))\bT(\bn_i)=\bT(\bv_1-\sum_{i=1}^m(\bn^T_i\bv_1)\bn_i)=\bT(\bP\bv_1)\,,
    \end{align}
   for any $\bv_1$, implying $
       \bT_n(\bN\bv_1,...,\bv_q)=(\bT_n(\bN\bv_1))(\bv_2,...,\bv_q)=(\bT(\bP\bN\bv_1))(\bv_2,...,\bv_q)=0$.
\end{proof}

\begin{proposition}[recursive algorithm]\label{recursive_algo}
     $\bT\in\mathbb{PT}_q$ if and only if
\begin{align}
     \bT^*\in \mathbb{PT}_{q-1}\,,\qquad \bT(\bn_i)=0\,,\quad 1\leq \forall i\leq m\,.
\end{align}
Also, the projection $\mathbb{P}\bT$ of a $\bT\in \mathbb{T}_q$ can be computed  by the recursive, two step procedure:
\begin{enumerate}
    \item  $\tilde{\bT}:=\{\mathbb{P}\bT^*\}\in \mathbb{T}_{q}$ and $\mathbb{P}f=f$\,,
    \item $\mathbb{P}\bT=\tilde{\bT}_n\in \mathbb{PT}_{q}$\,,
\end{enumerate}
with $\tilde{\bT}_n$  given by \eqref{t_n} with $\bT\equiv\tilde{\bT}$.
 \end{proposition}

\begin{proof} Assume $\bT\in\mathbb{PT}_q$ and observe for every  $1\leq i\leq m$ and for all $\bv_2$,..., $\bv_q$ that 
\begin{align*}
    \bT(\bn_i)(\bv_2,.,\bv_q)=\bT(\bn_i,\bv_2,.,\bv_q)=\bT(\bP\bn_i,\bv_2,.,\bv_q)=0\,,\quad 
\end{align*}
 and, consequently, $\bT(\bn_i)=0$. To demonstrate $\bT^*\in\mathbb{PT}_{q-1}$ one shows that for every $2\leq{}k\leq{}q$ and each $1\leq{}i\leq m$,
\begin{align*}
    \bT^*(\bv_2,.,\bv_{k-1},\bn_i,.,\bv_q)=\bT(\be^*)(\bv_2,.,\bv_{k-1},\bn_i,.,\bv_q)=\bT(\be^*,\bv_2,.,\bv_{k-1},\bn_i,.,\bv_q)=0\,.
\end{align*}
On the other hand, if $\bT\in\mathbb{T}_q$ satisfies  $\bT(\bn_i)=0$, for each $1\leq i\leq m$, then $\bT(\bN\bv_1)=0$ as
\begin{align*}
     \bT(\bn_i,\bv_2,.,\bv_q)&=\bT(\bn_i)(\bv_2,.,\bv_q)=0\,,
\end{align*}
 for all $\bv_1$. Since $\bT(\bv_1)=\sum_*\bv_1^*\bT^*\in\mathbb{PT}_{q-1}$ as  a linear combination of $\bT^*\in\mathbb{PT}_{q-1}$, we have 
\begin{align*}
    \bT(\bv_1,.,\bv_{k-1},\bn_i,.,\bv_q)= \bT(\bv_1)(\bv_2,.,\bv_{k-1},\bn_i,.,\bv_q)= \bT(\bv_1)(\bv_2,.,\bv_{k-1},\bP\bn_i,.,\bv_q)=0\,,
\end{align*}
 for any $2\leq{}k\leq{}q$. Hence, $\bT\in\mathbb{PT}_q$. 
The second claim regarding the projection $\mathbb{P}$ follows as, by the recursive construction, $\tilde{\bT}^*=\mathbb{P}\bT^*\in \mathbb{PT}_{q-1}$ satisfy the conditions of Lemma~\ref{lemma:t_n} so the final result of the procedure is tangent, $\tilde{\bT}_n\in \mathbb{PT}_{q}$, and 
\begin{align*}
    \tilde{\bT}_n(\bv_1,\bv_2,.,\bv_q)=\tilde{\bT}_n(\bP\bv_1,\bv_2,.,\bv_q)=\tilde{\bT}_n(\bP\bv_1)(\bv_2,.,\bv_q)
   =\tilde{\bT}(\bP\bv_1)(\bv_2,.,\bv_q)=\\ =(\textstyle{\sum_*}(\bP\bv_1)^*\tilde{\bT}^*)(\bv_2,.,\bv_q) 
   = \sum_*(\bP\bv_1)^*{\bT}^*(\bP\bv_2,.,\bP\bv_q)=\bT(\bP\bv_1,\bP\bv_2,.,\bP\bv_q)\,,
\end{align*}
where we used \eqref{t_n=t} in the first line and $\tilde{\bT}_n=\mathbb{P}\bT$ per Definition~\ref{def:tangency} as claimed. 
\end{proof}

\begin{proposition}\label{projection_with_inner}
    For any $\bT\in\mathbb{T}_q$ and any $\bS\in \mathbb{PT}_q$ we have 
    \begin{align}
    \bS\odot\bT=\bS\odot\mathbb{P}\bT\,.
    \end{align}
\end{proposition}
\begin{proof}
    For the base of induction $q=1$ one computes with a $\bu^T=\mathbb{P}\bu^T$ the following,
\begin{align}
    \bu^T\odot \bv^T=\bu^T(\bP\bv)=\bu\cdot(\bP\bv)=\bv^T(\bP\bu)= \mathbb{P}\bv^T(\bu)=  \bu^T \odot \mathbb{P}\bv^T\,.
\end{align}
Proposition~\ref{recursive_algo} yields the induction step, with the notation therein,
\begin{align}
\bS\odot\bT=\textstyle{\sum_*}\bS^*\odot\bT^* = \textstyle{\sum_*}\bS^*\odot\mathbb{P}(\bT^*)  = \textstyle{\sum_*} \bS^*\odot(\tilde{\bT})^*
=\textstyle{\sum_*} \bS^*\odot(\mathbb{P}\bT)^*=\bS\odot\mathbb{P}\bT\,,
\end{align}
because $(\mathbb{P}\bT-\tilde{\bT})^*=\sum_{i=1}^m\bn_i^*\otimes\tilde{\bT}(\bn_i)$ and, for each $1\leq i \leq m$,
\begin{align*}
    \textstyle{\sum_*}\bS^*\odot (\bn_i^*\tilde{\bT}(\bn_i))=(\textstyle{\sum_*}\bn_i^*\bS^*)\odot \tilde{\bT}(\bn_i)=\bS(\bn_i)\odot \tilde{\bT}(\bn)=0\,,
\end{align*}
since $\bS(\bn_i)=\bS(\bP\bn_i)=0$ for any tangential $\bS\in \mathbb{PT}_q$.
\end{proof}
 
\section{Tensor calculus on embedded submanifolds}\label{sec:calculus} 

We may attach an instance  of the  space of multilinear maps  from Definition~\ref{def:row-represent}  to each point $\bx \in M$ and denote it as $\mathbb{T}_q[\bx]$, $q\geq 0$. 
We identify these $\mathbb{T}_q[\bx]$ with a single $\mathbb{T}_q$ as they would take input arguments from the same domain $\mathbb{R}^n$.  This identification is not unique and the situation is analogous to the choice of parallel transport or of the connection in Riemannian geometry. Applications of this framework to continuum modeling may warrant different identifications, for example, involving tangent and normal to $M$ subspaces. However, as we will see, the forthcoming definitions of the corresponding derivation $\nablaM$ is directly connected with an extension of the extrinsic Stokes' formula of Proposition~\ref{vector_parts} to \textit{tensor fields} $\bT\in \mathbb{T}_q(M)$ of general rank $q$ defined as follows,
\begin{align}
    \bT\in \mathbb{T}_q(M) \qquad\Leftrightarrow \qquad  \bT[\bx]\in \mathbb{T}_q\,,\quad \forall \bx \in M\,.
\end{align}
In other words, an element  of the set $\mathbb{T}_q(M)$  assigns a map $\bT[\bx]\in  \mathbb{T}_q$ from Definition~\ref{def:row-represent} to every $\bx\in M$;  however,  the point specification $[\bx]$ in $\bT[\bx]$ will often be omitted.

\subsection{Differentiation and product rules} 
We start with the notion of parallel transport of tensor fields along one-dimensional embedded paths $M\equiv \gamma$, $n-m=1$, which we distinguish from curves that are parameterized. Note that we understand the limit below in the sense of the common space $\mathbb{T}_q$.
\begin{definition}[Cartesian parallel transport along a path] \label{def:cartesian}Consider a path $\gamma\subset\mathbb{R}^n$, a point $\bx\in \gamma$ and a vector $\bw$ tangent to $\gamma$ at $\bx$. Define $\mathcal{D}^\gamma_\bw\bT\in\mathbb{T}_q(\gamma)$ for $\bT\in\mathbb{T}_q(\gamma)$  by
     \begin{align*}
    \mathcal{D}^\gamma_\bw \bT [\bx] := \lim_{t\rightarrow{}0+}(\bT[\gamma(t)]-\bT[\bx])/t \,,
\end{align*}
where $\gamma(t)$ is any parametrization of $\gamma$ by $t\in I\subset \mathbb{R}$, $0\in  I$, such that $\gamma(0)=\bx$, $\pa_t \gamma(0)=\bw$.
\end{definition} 

The defining property in the following relies on Definition~\ref{def:left_right_insertions} of the right-insertion with $\bu$. Notice how $\nablaM$ disregards the normal component, $\nablaM\bT \cdot \bu = \nablaM\bT \cdot \bP\bu$, of any $\bu$, however $\nablaM\bT$ is not necessarily tangent per Definition~\ref{def:tangency}.
\begin{definition}[external derivative]\label{nablaM} Define $\nablaM\bT \in \mathbb{T}_{q+1}(M)$ for  a $\bT\in\mathbb{T}_q(M)$, $q\geq 0$, by 
    \begin{align*}
        \nablaM\bT[\bx] \cdot \bv := \mathcal{D}^\gamma_{\bP\bv} \bT[\bx] \,,\qquad\forall \bv\in \mathbb{R}^n\,,
    \end{align*}   
    where $\gamma \subset M$ is any path that $\bP\bv$ is tangent to.
\end{definition}

\begin{remark} Interesting enough, the above definition respects the immersion relation among submanifolds. Consider $\tilde{M}\subset M$ with the projector $\tilde{\bP}$, so that any path $\gamma{}\subset \tilde{M}$ is also a path in ${M}$. The external derivatives $\nabla_{\!\tilde{M}}\bT$ and $\nabla_{\!{M}}\bT$ of a given $\bT\in\mathbb{T}_q(M)$ differ on $\tilde{M}$ as tensors from $\mathbb{T}_{q+1}(\tilde{M})$, but nevertheless coincide pointwise as elements of $\mathbb{T}_{q-1}$ upon right-inserting,
\begin{align}\label{agreement_nablaM}
       \nabla_{\!\tilde{M}}\bT \cdot \bv = \mathcal{D}^\gamma_{\tilde{\bP}\bv} \bT = \mathcal{D}^\gamma_{{\bP}\bv} \bT  =\nablaM\bT\cdot \bv\,,
    \end{align}  
     with any $\bv\in\mathbb{R}^n$ such that $\tilde{\bP}\bv=\bP\bv$, or equivalently, for any tangential field $\bv^T\in \mathbb{PT}_1(\tilde{M})$.
    \end{remark}
\begin{proposition}\label{nablaM_rowrep}The row representation of external derivatives
   is given recursively by \eqref{basics} and
    \begin{align*}
         \nablaM \bT &= \{ \nablaM (\bT^*)\}\,.
    \end{align*}
\end{proposition}
\begin{proof}   By Definitions~\ref{def:cartesian} and \ref{nablaM} we have    for any $\bu$ that
\begin{align*}
   \nablaM\bT\cdot\bu= \mathcal{D}^\gamma_{\bP\bu} \bT =\lim_{t\rightarrow{}0}\{(\bT^*[\gamma(t)]-\bT^*[\bx])/t\}=\{\mathcal{D}^\gamma_\bw (\bT^*)\}=\{\nablaM(\bT^*)\cdot\bu\}\,,\qquad 
\end{align*}
and $(\nablaM\bT\cdot\bu)^*=\nablaM(\bT^*)\cdot\bu$. Proposition~\ref{insertion_rowrep}  and \eqref{star_component} imply  $(\nablaM\bT)^*=\nablaM(\bT^*)$.
\end{proof}

 \begin{definition} \label{def:divergence}
  Divergence  $\DivM\bT\in \mathbb{T}_{q-1}(M)$ for $\bT\in \mathbb{T}_q(M)$   is defined recursively in
 \begin{align*}
    \DivM\bT &:=\{\DivM \bT^*\}\,,\qquad  & \DivM\bu^T &:= \divM\bu\,,
    \\
   \CurlG \bT &:= \{\CurlG\bT^*\} \,,\qquad  &  \CurlG\bu^T &:= \curlG\bu\,,
 \end{align*}
 where  the curl  $\CurlG\bT\in \mathbb{T}_{q-1}(\Gamma)$ is also defined for a $\bT\in \mathbb{T}_q(\Gamma)$ only in case of  $n-m=2$.
 \end{definition}
 From this point on, an expression with $\CurlG\bT$ for a $\bT\in \mathbb{T}_q(M)$ always assumes that $M\equiv\Gamma$ is two-dimensional, $n-m=2$, within this expression. The contractions below are given in Definition~\ref{def:left_right_contraction}.
\begin{proposition}\label{product_rules} The following identities hold for any $\bT\in\mathbb{T}_{q}(M)$, $\bS\in\mathbb{T}_{s}(M)$, $q\geq s$,
\begin{align*}
    \nablaM(\bT\odot\bS)&=  \bS:\nablaM\bT +\bT: \nablaM\bS 
    \,,\qquad & \nablaM{}(\bS\otimes\bT)&= \nablaM\bS\otimes \bT+\bS\otimes \nablaM\bT\,,
\\
    \DivM(\bS:\bT) &= \bS:\DivM\bT+\bT:\nablaM\bS  \,,\qquad &  
    \CurlG(\bS:\bT) &= \bS:\CurlG \bT+\bT:\CurlGvec{}\bS \,,
\\
        \CurlG \nablaM\bT &= 0
         \,,\qquad & &  \DivM\CurlGvec\bT =0\,.
    \end{align*}
    such that they are well-defined. Here $\CurlGvec \bT\in \mathbb{T}_{q+1}(\Gamma)$ is defined recursively by 
    \begin{align}\label{def:curl}
    \CurlGvec \bT &:= \{\CurlGvec\bT^*\} \,,\qquad  &  \CurlGvec f &:= (\curlGscal{}f)^T\,.
    \end{align}
    \end{proposition}
    \begin{proof} We will prove the identities for all relevant $q$ and $s$ using induction and recursive definitions.
    Base of induction for $s=0$ and for the smallest values of $q$ that makes sense of identities:
         \begin{align*}
    \nablaM (f\otimes g)&=\nablaM (f\odot g)=\nablaM(fg)=g\nablaM f+f\nablaM g =\nablaM f \otimes g+f\otimes \nablaM g\,,
    \\
     \DivM(f:\bu^T) &=\divM(f\bu) = f\divM \bu+\nablaM{}f\cdot \bu= f:\DivM \bu^T+\bu^T:\nablaM{}f \,,
        \end{align*}
    by product rules from \cite{bouck_hydrodynamical_2024} and the symmetry of the inner product in Definition~\ref{def:inner_product}. We continue by recalling Definition~\ref{def:vector_scalar_curls},
    \begin{align*}
     \CurlG(f:\bu^T)&=\curlG(f\bu)  =-f\divM \bu^\dagger-\nablaM f (\bu^\dagger) = f\CurlG\bu^T + \bu^T:\CurlGvec f \,.
            \end{align*}
    Lastly, Corollary~\ref{cor:vanishing_div_curl}  completes the base of the induction,
    \begin{align*}
        \CurlG \nablaM f &=\curlG ((\nablaM f)^T)= 0
\,,\qquad
        \DivM(\CurlGvec f)=\divM(\curlGscal{}f)=0\,.
        \end{align*}
        Induction step with $q=s$  proves the first identity  as the rank of $\nablaM\bS$ is greater than of $\bT$,
        \begin{align*}
           \nablaM(\bT\odot\bS) &= \nablaM(\textstyle{\sum_*}\bT^*\odot\bS^*)= \textstyle{\sum_*}\nablaM(\bT^*\odot\bS^*) 
         = \textstyle{\sum_*}  (\bS^*):\nablaM(\bT^*)+(\bT^*): \nablaM(\bS^*)
         \\
                  &= \textstyle{\sum_*}  (\bS^*):(\nablaM\bT)^* +(\bT^*): (\nablaM\bS)^*=  \bS:\nablaM\bT +\bT: \nablaM\bS\,.
        \end{align*}
   Induction step over $q$ 
    for a fixed $s=0$ reads,
    \begin{align*}
   \nablaM (f\otimes \bT)&=\{\nablaM(\bT^*\otimes f)\}=\{\nablaM(\bT^*)\otimes f+\bT^*\otimes{}\nablaM f\}=\nablaM\bT\otimes{}f+\bT\otimes\nablaM f \,,
   \\
     \DivM(f:\bT) &= \{\DivM{}(f\bT)^*\} =\{f:\DivM \bT^*+\bT^*:\nablaM{}f\}
     = f:\DivM \bT+\bT:\nablaM{}f\,,
     \\
      \CurlG(f:\bT) &= \{\CurlG{}(f\bT)^*\} =\{f:\CurlG \bT^*+\bT^*:\CurlGvec{}f\}
     = f:\CurlG \bT+\bT:\CurlGvec{}f\,,
     \\
        \CurlG \nablaM\bT&=\{\CurlG (\nablaM\bT)^*\}=\{\CurlG \nablaM(\bT^*)\} = 0\,,
        \\
         \DivM(\CurlGvec \bT) &= \{\DivM(\CurlGvec \bT)^*\} = \{\DivM\CurlGvec (\bT^*)\}=0\,.
        \end{align*}
        Induction step over $0\leq{}s\leq{}q$ for a fixed but arbitrary $q$,
    \begin{align*}
     \nablaM (\bS\otimes \bT)&=\{\nablaM(\bS^*\otimes\bT)\}=\{\nablaM(\bS^*)\otimes\bT+\bS^*\otimes{}\nablaM \bT\}=\nablaM\bS\otimes{}\bT+\bS\otimes\nablaM \bT\,,
     \\
      \DivM(\bS:\bT) &= \DivM(\textstyle{\sum_*}\bS^*:\bT^*)=\textstyle{\sum_*}\DivM(\bS^*:\bT^*) = \textstyle{\sum_*}\bS^*:\DivM\bT^*+\bT^*:\nablaM\bS^* \,,
     \\
      \CurlM(\bS:\bT) &= \CurlM(\textstyle{\sum_*}\bS^*:\bT^*)=\textstyle{\sum_*}\CurlM(\bS^*:\bT^*) = \textstyle{\sum_*}\bS^*\!:\!\CurlM\bT^*+\bT^*\!\!:\!\!\CurlMvec\bS^* \,,
        \end{align*}
        completes the proof thanks to recursive definitions.
    \end{proof}

\subsection{Integration and extrinsic Stokes' formula} 
\begin{definition}[component-wise integration] \label{def:integral} For any $\bT\in\mathbb{T}_q(M)$, $q\geq 1$, define recursively 
     \begin{align*}\int_M \bT &= \{\int_M \bT^*\} 
     \end{align*}
    based on the same integral as in \eqref{generalStokes}.
 \end{definition}
 The recursive definition  above applied for one-dimensional paths $M\equiv \gamma$, with a tangent unit vector field $\bw$, yields the identity $\int_\gamma \mathcal{D}_\bw^\gamma \bT = \bT|_\ba^\bb$ since \eqref{FTC} holds for scalar fields. The following proposition demonstrates this within a somewhat generalized context of path integrals inside a given submanifold $M$.
 \begin{proposition} Consider a path $\gamma \subset M$ oriented with a tangent unit vector field $\bw$ with the start  $\ba$ and  the end $\bb$ points. We have for any $\bT\in\mathbb{T}_q(M)$, $q\geq 0$, 
 \begin{align}
     \int_{\gamma} \nablaM\bT\cdot\bw = \bT[\bb]-\bT[\ba] \,,\qquad  \qquad \bw\cdot\bw=1 \,.
 \end{align}
 \end{proposition}
\begin{proof} Prove by induction over $q\geq 0$. The base of induction follows from \eqref{FTC}  and \eqref{agreement_nablaM},
\begin{align*}
    \int_{\gamma} (\nablaM f)\cdot\bw= \int_{\gamma} (\nabla_\gamma f)\cdot \bw= f|_\ba^\bb\,,
\end{align*}
since $\gamma$ is immersed into $M$. 
The induction step completes the proof  as follows,
    \begin{align*}
    \int_{\gamma} \nablaM\bT \cdot \bw=  \{ \int_{\gamma} (\nablaM\bT\cdot\bw)^*  \}= \{ \int_{\gamma} \nablaM(\bT^*)\cdot\bw \}=\{\bT^*(\bb)-\bT^*(\ba)\}=\bT(\bb)-\bT(\ba)\,,
\end{align*}
where we recall  the recursive Definition~\ref{def:integral} along with Propositions~\ref{insertion_rowrep} and \ref{nablaM_rowrep}.
\end{proof}

  \begin{proposition}[Stokes' formula]\label{comp_tensor_parts}
 For any $\bT\in\mathbb{T}_q(M)$, $q\geq 1$, we have
\begin{align*}
\int_M \DivM\bT  &=\int_{\pa{}M}\bT\cdot\bt+ \int_M \bT\cdot\bm{\kappa}\,.
\end{align*}
Also, in case of  $n-m=2$, for any $\bT\in\mathbb{T}_q(\Gamma)$, $q\geq 1$,
\begin{align*}\int_\Gamma \CurlG\bT  &= \int_{\pa\Gamma} \bT\cdot\bm{\tau}\,.
\end{align*}
\end{proposition}
\begin{proof} The base $q=1$ of the induction argument is shown in Proposition~\ref{vector_parts}. Definitions~\ref{def:divergence} and \ref{def:integral} followed by the induction step and Proposition~\ref{insertion_rowrep} yield the first claim,
\begin{align*}
\int_M \DivM\bT = \{\int_M   (\DivM\bT)^* \}= \{\int_M \DivM(\bT^*)\}
=\{\int_{\pa{}M}\bT^*\cdot\bt+ \int_M  \bT^*\cdot{}\bm{\kappa}\}
\\
=\{\int_{\pa{}M}(\bT\cdot\bt)^*\}+ \{\int_M  (\bT\cdot\bm{\kappa})^*\}
=\int_{\pa{}M}\bT\cdot{}\bt+ \int_M  \bT\cdot\bm{\kappa}\,.
    \end{align*}
   Similarly, Proposition~\ref{curl_int_parts} as the base and the following induction step complete the proof,
    \begin{align*}
\int_\Gamma \CurlG\bT = \{\int_\Gamma   (\CurlG\bT)^* \}= \{\int_\Gamma \CurlG(\bT^*)\}
=\{\int_{\pa\Gamma}\bT^*\cdot\bm{\tau} \}
=\int_{\pa\Gamma}\bT\cdot\bm{\tau}\,.
    \end{align*}
\end{proof}
 
\begin{corollary}[integration by parts]\label{integration_by_parts}
  For $\bT\in\mathbb{T}_q(M)$ and $\bS\in\mathbb{T}_{s}(M)$, $s < q$ we have
\begin{align*}
  \int_M \bS:\DivM\bT &= -\int_{M} \bT:\nablaM{}\bS +\int_{\pa{}M} (\bS:\bT)\cdot{}\bt +\int_{M} (\bS:\bT)\cdot\bm{\kappa}\,.
  \\
    \int_\Gamma \bS:\CurlG\bT &= - \int_\G \bT:\CurlGvec \bS+\int_{\pa\Gamma } (\bS:\bT)\cdot\bm{\tau}\,.
\end{align*}
\end{corollary}
\begin{proof}
    Implication of product rules in Proposition~\ref{product_rules} , \eqref{def:curl} and Proposition~\ref{comp_tensor_parts} for $\bS:\bT$.
\end{proof}

\subsection{Covariant derivative}\label{sec:covariant_derivative}
We now discuss how the tensor calculus on embedded submanifolds  introduced in Section~\ref{sec:calculus} relates to the intrinsic approach that involves Levi-Civita connection and the associated covariant tensor derivative on a Riemannian manifold. 

 \begin{definition}[covariant derivative]\label{def:covariant_derivative} For a $\bT\in \mathbb{T}_q(M)$, $\nablaM^{\mathrm{cov}} \bT\in\mathbb{PT}_{q+1}(M)$ is defined by
 \begin{align}
     \nablaM^{\mathrm{{cov}}} \bT &:= \mathbb{P}(\nablaM \bT)\,.
 \end{align}
 \end{definition}

If $\bT\in \mathbb{PT}_q$ then $\nablaM^{\mathrm{{cov}}}\bT\in \mathbb{PT}_{q+1}$ is a tangent tensor field that can be understood intrinsically to $M$ and what follows is an elaboration on this point in the case of $q=1$. The cornerstone identities \eqref{directional_product} and \eqref{directional_product_cov} that correspond to covariant and extrinsic tensor calculi emanate from  Proposition~\ref{product_rules} and \eqref{vector-covector_insert} as follows,
\begin{align}\label{scalar_separates_in_vector_grad}\nablaM(\bu^T(\bv))=(\nablaM\bu^T)(\bv)+(\nablaM\bv^T)(\bu)\,,\qquad \forall \bu^T\,, \bv^T \in \mathbb{T}_1(M)\,,
\end{align}
where only tangential terms are involved a priori: a submanifold gradient of a scalar field $\bu^T(\bv)$ on the left hand side, and linear combinations of submanifold gradients of scalar components $\bu^*$ and  $\bv^*$, see Definition~\ref{insertion_rowrep}. Contracting \eqref{scalar_separates_in_vector_grad} with a $\bw\in \mathbb{T}_1(M)$ and flipping the order using Proposition~\ref{insertion_rowrep} as in, for instance, $((\nablaM\bu^T)\cdot\bw)(\bv)= ((\nablaM\bu^T)(\bv))\cdot\bw$,  give
\begin{align}\label{directional_product}
\nablaM(\bu^T(\bv))\cdot\bw =((\nablaM\bu^T)\cdot\bw)(\bv)+((\nablaM\bv^T)\cdot\bw)(\bu)\,,\qquad \forall \bu^T\,, \bv^T\,,\bw^T \in \mathbb{T}_1(M)\,.
\end{align}
 As is, \eqref{directional_product} shows that $\nablaM$ is a proper derivation on $\mathbb{T}_1(M)$, see, e.g., \cite[Chapter III]{kobayashi1996foundations}.   One may wish to restrict \eqref{directional_product} to tangential $\bu=\bP\bu$ and $\bv=\bP\bv$ and, hence, to tangential  $\bu^T=\mathbb{P}\bu^T$ and $\bv^T=\mathbb{P}\bv^T$ and write
\begin{align}\label{directional_product_with_tangents}
\nablaM(\bu^T(\bP\bv))\cdot(\bP\bw) =((\nablaM\bu^T)\cdot(\bP\bw))(\bP\bv)+((\nablaM\bv^T)\cdot(\bP\bw))(\bP\bu)\,.
\end{align}
while also replacing $\bw=\bP\bw$ per Definition~\ref{nablaM}.
Every object on the left hand side is tangential.
Yet the tensor field $\nablaM\bu^T$ in \eqref{scalar_separates_in_vector_grad} is not tangential even for a tangential $\bu^T$: there exists a $\bv$ such that $(\nablaM\bu^T)(\bv) \neq (\nablaM\bu^T)(\bP\bv)$ even though both sides here are automatically tangential as articulated above. The covariant derivative of Definition~\ref{def:covariant_derivative} fixes this by removing the non-tangent part of $\nablaM$ operator so $(\nablaM^{\mathrm{cov}}\bu^T)(\bv) = (\nablaM^{\mathrm{cov}}\bu^T)(\bP\bv)$ holds and, if restricted to tangential $\bu^T$, maps $\mathbb{PT}_1(M)$ to $\mathbb{PT}_{2}(M)$. To see this, one rewrites \eqref{scalar_separates_in_vector_grad} by restricting it to $\bu=\bP\bu$, $\bv=\bP\bv$, $\bu^T=\mathbb{P}\bu^T$ and $\bv^T=\mathbb{P}\bv^T$, to get $\nablaM(\bu^T(\bv))=(\nablaM\bu^T)(\bP\bv)+(\nablaM\bv^T)(\bP\bu)=(\nablaM^{\mathrm{cov}}\bu^T)(\bv)+(\nablaM^{\mathrm{cov}}\bv^T)(\bu)$ and obtains similarly to \eqref{directional_product} that 
\begin{align}\label{directional_product_cov}
\nablaM^{\mathrm{cov}}(\bu^T(\bv))\cdot\bw=((\nablaM^{\mathrm{cov}}\bu^T)\cdot\bw)(\bv) +((\nablaM^{\mathrm{cov}}\bv^T)\cdot{}\bw)(\bu)\,, \qquad \forall \bu^T\,, \bv^T\,,\bw^T \in \mathbb{PT}_1(M)\,.
\end{align}
Consequently, $\nablaM^{\mathrm{cov}}$ is a derivation operator that, unlike $\nablaM$, leads to an intrinsic description.  In what follows, we compare these derivation operators through the lens of distinct tensors Laplacians on $M$ that correspond to them. For scalar fields $f$, $\DivM(\nablaM f)=\DivM(\nablaM^\mathrm{cov} f)$ is the well-known Laplace-Beltrami operator on $M$. Note that we do not introduce $\DivM^{\mathrm{cov}}$ as it coincides with $\DivM$ due to  the cyclic property of traces in \eqref{basics}. 

 \subsection{Tensor Laplace-Beltrami operator}
 In the remaining part, we introduce and discuss two distinct tensor, elliptic operators in Definition~\ref{def:laplacians}.
  \begin{definition}\label{def:laplacians}
  The covariant Laplacian  is defined along with the extrinsic Laplacian as follows
     \begin{align*} 
\Delta_M^{\mathrm{\mathrm{cov}}}\bT:=\mathbb{P}\DivM\nablaM^{\mathrm{\mathrm{cov}}}\bT \,, \qquad  \qquad \qquad \Delta_M\bT& :=\DivM\nablaM\bT\,.
        \end{align*}
\end{definition}
In principle, the covariant Laplacian $\Delta_M^{\mathrm{\mathrm{cov}}}$ can be understood intrinsically as in the above discussion. The external Laplacian $\Delta_M$ can be regarded as trivial even for a general $\bT\in\mathbb{T}_q(M)$. Indeed, since Definition~\ref{def:divergence} and Proposition~\ref{nablaM_rowrep} are recursive,  $\DivM\nablaM{}\bT$  is a component-wise operator, i.e. $\DivM\nablaM\bT=\{\DivM(\nablaM(\bT^*))\}$, which trickles down to applications of the scalar Laplace-Beltrami operator to $n^q$ scalar components of a $\bT\in\mathbb{T}_q(M)$. 

Consider the following elliptic, tensor problem: Given a forcing $\bbf \in \mathbb{T}_q(M)$ and a flux $\bq \in \mathbb{T}_q(\pa{}M)$, find $\bT\in\mathbb{PT}_q(M)$ such that
 \begin{align}\label{LaplaceM_cov}
-\Delta_M^{\mathrm{\mathrm{cov}}}\bT&=\mathbb{P}\bbf\quad\mathrm{on}\quad M\,,& \qquad  \nablaM^{\mathrm{cov}}\bT\cdot\bt &= \mathbb{P}\bq \quad\mathrm{on}\quad \pa{}M\,,
 \end{align}
 where $\bt$ is the co-normal as in \eqref{generalStokes}. Multiplying the left-hand side of \eqref{LaplaceM_cov} by $\bS\in \mathbb{PT}_q(M)$, integrating by parts through Corollary~\ref{integration_by_parts} and utilizing Proposition~\ref{associativity} and \ref{projection_with_inner} give
 \begin{align*}
   -\int_M \bS\odot\mathbb{P}\bbf &= \int_M \bS\odot\mathbb{P}\DivM\nablaM^{\mathrm{cov}}\bT = -\int_{M} \nablaM^{\mathrm{cov}}\bT\odot\nablaM^{\mathrm{cov}}\bS +\int_{\pa{}M } \bS\odot\mathbb{P}\bq +\int_{M} \bS\odot(\nablaM^{\mathrm{cov}}\bT\cdot\bm{\kappa})\,.
 \end{align*}
  Note that $\nablaM\bT\cdot\bm{\kappa}=\nablaM^{\mathrm{cov}}\bT\cdot\bm{\kappa}=0$ by Definition~\ref{def:covariant_derivative} as $\bP\bm\kappa=0$. Consequently,  any classical solution $\bT\in\mathbb{PT}_q(M)$ of the problem \eqref{LaplaceM_cov} satisfies  
  \begin{align}\label{weak_LB}
      a^{\mathrm{cov}}(\bT,\bS)=\ell(\bS)\,,\qquad  \qquad  \qquad  \forall \bS\in\mathbb{PT}_q(M)
  \end{align}
with $a^{\mathrm{cov}}(\bT,\bS):=\int_{M} \nablaM^{\mathrm{cov}}\bT\odot\nablaM^{\mathrm{cov}}\bS$ and $\ell(\bS):=\int_{\pa{}M}\bS\odot\bq+\int_M \bS\odot\bbf
$. Consequently, the notion of weak solutions can be conceived from \eqref{weak_LB} as well as computational methods for solving the tensor problem  \eqref{LaplaceM_cov}.

\section{Application: Euler equations and external momentum}\label{sec:Euler_eq}

We consider well-known intrinsic Euler equations on a fixed manifold that is formulated in \eqref{Euler-non-divergence} using the notation developed in Section~\ref{sec:covariant_derivative}. The solution of \eqref{Euler-non-divergence}  is an incompressible flow with a tangential velocity $\bu$ that corresponds to the covector field $\bu^T$. Since \eqref{Euler-non-divergence} is posed on an embedded submanifold $M$, we  may consider the following extrinsic momentum assuming a constant density, $\rho=\mathrm{const}$,
\begin{align}
\label{extrinsic_momentum}\bJ\equiv\bJ[\bu^T]:=\int_M \rho\bu^T \quad \in \quad \mathbb{R}^n \qquad \mathrm{for} \qquad \bu^T\in \mathbb{T}_1(M)\,.
    \end{align}
   Newton laws of mechanics dictate that the quantity $(-\bJ)$ equals the total reaction force that must be somehow applied to the embedded submanifold since it stays fixed in $\mathbb{R}^n$. The corresponding local reaction forces emerge on the boundary $\pa{}M$ and on $M$ itself but are often eliminated from consideration. The next lemma is instrumental in showing that, actually, $\bJ=0$ for any tangential flow $\bu$ on $M$ as long as it is incompressible and tangential to the boundary $\pa{}M$. Consequently,  such a flow on any embedded submanifold never results in a non-zero global, net reaction force regardless of equations that govern it. However, this observation does not say that  local reaction forces for material patches of the embedded submanifold vanish -- only that they must cancel out globally. Note that in what follows, the integrals are vectors in $\mathbb{R}^n$ and taken component-wise.
\begin{lemma}\label{lem:tangent_velocity}For a vector field $\bu$ on an embedded submanifold $M$ with boundary $\pa{}M$ we have
 \begin{align*}
      \int_M \bP\bu = -\int_M \divM(\bP\bu)\br+\int_{\pa{}M} (\bu\cdot\bt)\br
  \end{align*}
  with $\br$ and $\bt$ given in Definition~\ref{mean_curvature_vector} and \eqref{generalStokes}.
\end{lemma}
\begin{proof}Definition~\ref{def:integral}, Proposition~\ref{insertion_rowrep}, \eqref{vector-covector_insert} and integration by parts via Corollary~\ref{integration_by_parts} result in
\begin{align*}
 \int_M (\mathbb{P}\bu^T)^* = \int_M \be^*:\mathbb{P}\bu^T= \int_M\nablaM(\br^*):\mathbb{P}\bu^T= -\int_{M}\br^*\divM(\mathbb{P}\bu^T)+\int_{\pa{}M}\br^*(\mathbb{P}\bu^T\cdot\bt)\,,
  \end{align*}
  where Proposition~\ref{projection_with_inner} was utilized in the second equality and, for the last one, that $\bP\bm\kappa=0$.
 \end{proof}
\begin{corollary}\label{cor:zero_momentum}
    Any incompressible, tangential flow $\bu$ on $M$ with constant density $\rho=\rho_0$ that is also tangential to the boundary $\pa{}M$ has zero extrinsic momentum \eqref{extrinsic_momentum}, i.e.
    \begin{align*}
        \bJ[\bu] =0\,\qquad \mathrm{for}\qquad  \forall \bu \quad\mathrm{s.t}  \quad \bu=\bP\bu\,, \quad \divM\bu=0\,,\quad \bu\cdot\bt=0\,.
    \end{align*}
\end{corollary}

\subsection{Euler equations}
We would like to emphasize that the previous statement is not unique to Euler equation and this is why it is introduced only now as follows: Find a tangential velocity $\bu^T(t)\in \mathbb{PT}_1(M)$ and scalar pressure $p\in \mathbb{T}_0(M)$ such that for all $t\in (0,T)$,
\begin{align}\label{Euler-non-divergence}
     \pa_t{\bu^T} +(\nablaM^\mathrm{cov}\bu^T)\cdot \bu &=  -\nablaM p\,, \qquad \divM\bu=0  \quad \mathrm{on}\quad {M}\,,\qquad  \qquad \bu\cdot\bt =0 \quad \mathrm{on}\quad \pa{M} \,,
 \end{align}
given an initial condition $\bu^T(0)$. We can also recast \eqref{Euler-non-divergence} into the divergence form via the following identity,
\begin{align}\label{div-convective_indetity}
(\nablaM^\mathrm{cov}\bu^T)\cdot\bu=\mathbb{P}\DivM(\bu^T\otimes \bu^T)\qquad \mathrm{for}\qquad  \divM\bu =0\,,
\end{align}
which can be deduced by recalling Definitions~\ref{def:covariant_derivative}, \ref{def:divergence},  Propositions~\ref{nablaM_rowrep}, \ref{outer_represent} and \eqref{curl_product} in
\begin{align*}
&\DivM(\bu^T\otimes \bu^T)=\{\DivM(\bu^*\bu^T)\}=\{\DivM(\bu^*\bu^T)\}=
\{\bu^*\DivM \bu^T + \nablaM(\bu^*)\odot{} \bu^T\}
\\&=\DivM \bu^T\{\bu^*\}+ \{(\nablaM\bu^T)^*\odot{} \bu^T\}=(\divM\bu)\bu^T+ (\nablaM\bu^T):\bu^T=(\divM \bu)\bu^T+ \nablaM\bu^T\cdot\bu\,.
\end{align*}
Note that $\nablaM \bu^T=\nablaM^\mathrm{cov}\bu^T - \sum_{j=0}^k \bB_j(\bu) \otimes \bn^T_j$ and the application of $\mathbb{P}$  in \eqref{div-convective_indetity} is not redundant. Also, $\DivM\bP=-\bm\kappa^T$ from \eqref{mean_curvature_vector_components}  and \eqref{projection-row-represent} and, by Propositions~\ref{product_rules} and \ref{vector_parts},
\begin{align*}
    \mathbb{P}\DivM(p\bP)=\mathbb{P}(-p\bm\kappa^T+\bP:\nablaM p)=\nablaM p \,.
\end{align*}
Consequently, the equation in \eqref{Euler-non-divergence} can be recast in the divergence form, 
 \begin{align}\label{Euler-divergence}
     \pa_t{\bu}^T +\mathbb{P}\DivM(\bu^T \otimes\bu^T + p\bP) =0\,.
 \end{align}

 \subsection{Ambient force balance}
We now formally take the time derivative of $\bJ(t)=0$ that holds due to Corollary~\ref{cor:zero_momentum},
   \begin{align}
   \label{vanishing_momentum}0=\frac{1}{\rho}\frac{d}{dt}\bJ &= \int_M \pa_t\bu^T=-\int_M (\nablaM{}p+(\nablaM^\mathrm{cov}\bu^T)\cdot\bu) =  -\int_M\mathbb{P}\DivM\bq\,,
 \end{align}
  where we utilized \eqref{Euler-divergence} and introduced the flux $\bq:=\bu^T \otimes\bu^T + p\bP\in\mathbb{PT}_2(M)$. Due to  the projector $\mathbb{P}$ in \eqref{vanishing_momentum} we need to rely on Corollary~\ref{integration_by_parts} as in $\int_M \bT = \{\int_M \bT^*\}=\{\int_M \bT(\be_*)\}=\{\int_M \bT:\be^*\}$,
 \begin{align*}
   \int_M (\mathbb{P}\DivM\bq):\be^* = \int_M \mathbb{P}\be^* \odot \DivM\bq
=
   -\int_{M} \bq:\nablaM{}(\mathbb{P}\be^*) +\int_{\pa{}M} ((\mathbb{P}\be^*):\bq)\cdot{}\bt +\int_{M} ((\mathbb{P}\be^*):\bq)\cdot\bm{\kappa}\,,
\end{align*}
and the last term vanishes due to Proposition~\ref{associativity} that also implies $((\mathbb{P}\be^*):\bq)\cdot{}\bt=(\mathbb{P}\be^*):(\bq\cdot{}\bt)=p(\be^*\odot \bt)=p\bt^*$ per the boundary condition in \eqref{Euler-non-divergence}. The remaining term can be treated  as follows,
\begin{align*}
    -\nablaM{}(\mathbb{P}\be^*)= \nablaM{}(\bN\be^*)=\nablaM \left(\sum_{i=0}^m\bn_i^* \bn^T_i \right)=\sum_{i=0}^m\nablaM\left( \bn_i^* \bn^T_i \right)=\sum_{i=0}^m\left(\bn_i^*(\nablaM  \bn^T_i)+ (\nablaM\bn_i^*)\otimes\bn^T_i\right)\,.
    \end{align*}
    Hence, Proposition~\ref{projection_with_inner} implies that the following tensor is a linear combination of $\bn^T_i$, $1\leq i\leq m$,
    \begin{align*}
    \{\bq\odot(-\nablaM{}(\mathbb{P}\be^*))\} = \{\sum_{i=0}^m\bn_i^*\left( \bq\odot(\nablaM  \bn^T_i)\right)\}=\sum_{i=0}^m\left( \bq\odot\bB_i\right)\bn_i^T = p\bm\kappa+\sum_{i=0}^m\bu^T:\bB_i(\bu)\bn_i^T \,,
\end{align*}
where we inserted $\bq\odot\bB_i = \bu^T:\bB_i(\bu)+p\bP\odot\bB_i$ in the last equality and recalled \eqref{mean_curvature_vector_components}.
Thus, \eqref{vanishing_momentum} yields the following result.
\begin{theorem}\label{th:force_balance} Any sufficiently regular Euler flow \eqref{Euler-non-divergence} on a fixed Riemannian manifold $M$ satisfies the following extrinsic conservational principle for all times: the Young-Laplace force $p\bm\kappa$, the boundary reaction $p\bt$ are balanced by the centripetal force, i.e.
\begin{align*}
      \int_M p\bm\kappa + \int_{\pa{}M}p \bt + \sum_{i=0}^m\int_M (\bB_i(\bu)\cdot\bu)\bn_i  =0 \,.
\end{align*}
\end{theorem}

    \begin{remark}\label{remark:JRO}
   A model of inextensible viscous fluidic two-dimensional film $\Gamma(t)$ evolving with the material velocity $\bu$ in $\mathbb{R}^3$ due to a given external force $\bb$ was introduced in \cite{jankuhn_incompressible_2018} by postulating the rate of change \cite[(3.5)]{jankuhn_incompressible_2018} of the extrinsic momentum $\bJ$ for any material patch $\tilde{\Gamma}\subset\Gamma$ that, if one sets viscosity $\mu=0$ and assumes a normal force $\bb=b_N\bn$, takes the following form
   \begin{align}\label{law_momentum}
    \frac{d}{dt}\int_{\tilde{\Gamma}(t)}\rho\bu = -\int_{\pa{}\tilde{\Gamma}(t)}p\bt+\int_{\tilde{\Gamma}(t)} b_N \bn \,.
    \end{align}
Also, within the model, it was found that the surface  stays stationary, $\Gamma(t)=\Gamma$ for all $t$, with a tangential material velocity, $\bu=\bP\bu$, if it holds at every point of $\Gamma$ that
\begin{align}\label{exact_stationary_force}
    \bb =- p\bm\kappa  - (\bB(\bu)\cdot\bu  )\bn\,.
\end{align}
Consequently, the model in \cite{jankuhn_incompressible_2018} is equivalent to the Euler problem \eqref{Euler-non-divergence} with $n=3$ and $m=1$ in case \eqref{exact_stationary_force}. The total momentum balance \eqref{law_momentum} for the entire surface $\Gamma$ with the boundary $\pa\Gamma$ gives 
\begin{align}\label{law_vanishing}
    \frac{d}{dt}\int_{{\Gamma}}\rho     \bu = -\int_{\pa{}{\Gamma}}p\bt-\int_{{\Gamma}}  p\bm\kappa  - \int_{{\Gamma}} (\bB(\bu)\cdot\bu  )\bn \,.
    \end{align}
Theorem~\ref{th:force_balance} implies that the right hand side of \eqref{law_vanishing} must be zero on a stationary surface for all times. In fact, Lemma~\ref{lem:tangent_velocity} shows that not only the rate $\frac{d}{dt}\bJ(t)$ on the left hand side of \eqref{law_vanishing} but the momentum $\bJ(t)$ vanishes, for the arbitrary dimension and codimension of the embedded submanifold $M$.
\end{remark}

\section{Application: the revisited concept of Cauchy stress on submanifolds}\label{sec:stress_submanifold}

Let us introduce the  Cauchy stress tensor assumed to act at points of a submanifold $M$.  We will accept the phenomenological approach in which arbitrary forces are allowed to act on generally oriented, possibly non-tangential planes. Hence, in the following definition, we postulate that the stress tensor field $\bsigma$ on a submanifold is connected via Proposition~\ref{insertion_rowrep} with a general force field $\bbf_\bv$.
\begin{definition}[Submanifold Cauchy stress]\label{def:Cauchy_stress}
    Given stress vectors $\bbf_\bv \in \mathbb{T}_1(M)$ that are associated with planes orthogonal to  general $\bv\in \mathbb{R}^n$, the stress tensor  $\bsigma=\{\bsigma^*\}\in\mathbb{T}_2(M)$ is characterized by 
    \begin{align*}
    \bsigma(\bv)= \bbf_\bv \,.
    \end{align*}
\end{definition}
The components $\bsigma^* \in \mathbb{T}_1(M)$ are conceptually equivalent to the stress vectors of the classical Cauchy stress which are  associated with the rows of the matrix of the stress tensor instead of its columns. We keep in mind that, with this association, the equation of equilibrium involves the transpose of the stress matrix. Due to the conservation of angular momentum,  the classical Cauchy stress tensor is symmetric and the distinction between rows and columns disappears. We now revisit this conclusion for Cauchy stresses on embedded submanifolds.

First we introduce the concept of torque tensor in $\mathbb{R}^n$ whose conservation represents  the rotational symmetry of $\mathbb{R}^n$ that complements the translation symmetry.  Consider the set  $\bE\equiv \bE_n$ of $n(n-1)/2$ unordered pairs of distinct vectors that can be formed from the $n$ basis elements $\be^*$. Every element $\bE_{ij}:=\{\be_i,\be_j\}\equiv \{\be_j,\be_i\}\in \bE$, $1\leq i < j\leq n$, is associated with the two-dimensional subspace spanned by $\be_i$ and $\be_j$. In other words, $\bE$ is the bivector basis of $\mathbb{R}^n$. In addition, there is a one-to-one correspondence between $\bE$ and the set of all infinitesimal generators of rotations  $\mathbf{l}_{ij}\in \mathbb{T}_1(M)$ of the plane $\bE_{ij}$. We may express these generators using $\br(\bx)$ from Definition~\ref{mean_curvature_vector}  in
\begin{align}\label{generators}
   \mathbf{l}_{K}\equiv \mathbf{l}_{ij}:=  \br^i \be_j^T -\br^j \be_i^T\,, \qquad 1\leq i < j\leq n\,,\qquad 1\leq K \leq n(n-1)/2 \,,
\end{align}
where the multi-index $K=\{i,j\}$ was introduced to stress that $\mathbf{l}_{ij}$ are not tensor components. A given point force $\bbf \in \mathbb{T}_1$ is said to generate torque $x_i \bbf{}^j -x_j \bbf{}^i\in \mathbb{T}_0$ with respect to the subspace $\bE_{ij}$. The torques with respect to all $\bE_{ij}\equiv \bE_{K}$ can be rewritten with the help of Definition~\ref{def:inner_product} as
\begin{align*}
   \mathbf{l}_{K}\odot\bbf\,, \qquad 1\leq K \leq n(n-1)/2\,,
\end{align*}
that will appear in the forthcoming integral definition.
We consider a material patch $\tilde{M}\subset M$ subjected to stress $\bsigma\in\mathbb{T}_2(M)$ as in Definition~\ref{def:Cauchy_stress} and the corresponding stress vector $\bbf_\bv = \bsigma(\bv)$ that generates the total force $\bF$ and the total torques $m_{K}$.
\begin{definition}[force and torque]\label{def:force-torque} Given $\tilde{M}\subset M$ and $\bsigma\in\mathbb{T}_2(M)$, define $\bF\in\mathbb{T}_1$, $m_{K}\in\mathbb{T}_0$ by
    \begin{align*}
 \bF&\equiv \bF[\bsigma,\tilde{M}]:=  \int_{\pa  \tilde{M}} \bsigma(\bt) + \int_{\tilde{M}} \bsigma(\bm{\kappa})\,,
 \\
  m_{K}&\equiv  m_{K}[\bsigma,\tilde{M}]:=\int_{\pa \tilde{M}} \mathbf{l}_{K}\odot \bsigma(\bt) +\int_{\tilde{M}} \mathbf{l}_{K}\odot{}\bsigma(\bm{\kappa}) \,,\qquad   1\leq K \leq n(n-1)/2 \,.
\end{align*}
Also,  $M$  is said to be in equilibrium under  $\bsigma$ if all the quantities above vanish for any $\tilde{M}\subset M$.
\end{definition}
Note that the stress vector $\bbf_\bv$ is integrated not only along the boundary of $\tilde{M}$ in the direction of $\bv=\bt$ but also along the $M$ itself in the fictitious direction given by the mean curvature vector,  $\bv=\bm\kappa$. We insist on this phenomenological choice, as omitting it in Definition~\ref{def:force-torque} requires the later introduction of external forces that is required for the equilibrium of $M$ to be possible. 

\subsection{Equilibrium of forces}
  Before deriving the equation of equilibrium, we introduce the transpose of a rank 2 tensor.
  \begin{definition}[transpose $\bar{\bT}$ of $\bT$] \label{def:transpose} Given $\bT\in\mathbb{T}_2$, define $\bar{\bT}\in\mathbb{T}_2$ by 
      \begin{align*}
      \bar{\bT}\cdot\bv = \bT(\bv)\,,\quad \forall \bv\in \mathbb{R}^n
      \qquad \Leftrightarrow  \qquad 
      \bar{\bT}=\{\bT:\be^*\}\,.
  \end{align*}
  \end{definition}
  
\begin{proposition}\label{div_sigma_equilib} If a submanifold $M\subset \mathbb{R}^n$ is in equilibrium under stress $\bsigma\in \mathbb{T}_2(M)$, then
  \begin{align*}
 \DivM\bar{\bsigma}&=0\,
\end{align*}  
where $\bar\bsigma \in \mathbb{T}_2(M)$ is the pointwise transpose from Definition~\ref{def:transpose}.
\end{proposition}
\begin{proof}
    Consider an arbitrary $\tilde{M}\subset M$ in equilibrium. Since  the force from Definition~\ref{def:force-torque} vanishes,
    \begin{align*}
        0 = \bF[\bsigma, \tilde{M}] =  \int_{\pa  \tilde{M}} \bar{\bsigma}\cdot\bt + \int_{\tilde{M}} \bar{\bsigma}\cdot\bm{\kappa}=  \int_{\tilde{M}} \DivM\bar{\bsigma}\,,
    \end{align*}
      where we used \eqref{def:transpose}, $\bar{\bsigma}\cdot\bv = \bsigma(\bv)$, and Stokes formula from Proposition~\ref{comp_tensor_parts}.
\end{proof}

  \subsection{Equilibrium of torques}
We start with a useful identity  that connects the equilibrium of torques with that of forces through the introduction of tensors $\bm\omega_{K}\equiv\bm\omega_{ij}$ indexed as in \eqref{generators}. \begin{lemma}\label{lemma:l_ij_A}For any $\tilde{M}\subset \mathbb{R}^n$  and any $\bA\in\mathbb{T}_2({\tilde{M}})$, it holds for every $1\leq K \leq n(n-1)/2$ that
    \begin{align*}
\int_{\pa{}\tilde{M}}(\mathbf{l}_{K}:\bA)\cdot\bt+ \int_{\tilde{M}}(\mathbf{l}_{K}:\bA) \cdot \bm{\kappa} = \int_{\tilde{M}}\mathbf{l}_{K}:\DivM\bA -  \int_{\tilde{M}} \bA\odot\bm{\omega}_{K}
\end{align*}
with  tensors  $\bm{\omega}_{K}\in\mathbb{T}_2(M)$ correspondent to \eqref{generators} and given by
\begin{align}\label{omegas}
\bm{\omega}_{K}\equiv\bm{\omega}_{ij}:=\be^i \otimes \bP^j-\be^j\otimes \bP^i \,.
\end{align}
\end{lemma}
\begin{proof}
Proposition~\ref{comp_tensor_parts} with $\bT=\mathbf{l}_{K}:\bA$ reads
\begin{align}\label{div_lij_A}
\int_{\tilde{M}} \DivM(\mathbf{l}_{K}:\bA)  &=\int_{\pa{}{\tilde{M}}}(\mathbf{l}_{K}:\bA)\cdot\bt+ \int_{\tilde{M}} (\mathbf{l}_{K}:\bA) \cdot\bm{\kappa}\,.
\end{align}
At the same time, using the product rules of Proposition~\ref{product_rules} we deduce for any $1\leq i\leq n$,
\begin{align*}
  \DivM((\br^i\be^*):\bA) =(\br^i\be^*):\DivM\bA+  \bA\odot(\be^*\otimes\bP^i)\,,
\end{align*}
since $ \nablaM(\br^i\be^*)=\br^i\nablaM\be^*+\be^*\otimes\nablaM(\br^i) = \be^*\otimes\bP^i$. 
Skew-symmetrizing as in \eqref{generators} gives
\begin{align*}
    \DivM(\mathbf{l}_{ij}:\bA) =\mathbf{l}_{ij}:\DivM\bA+  \bA\odot(\be^j\otimes\bP^i-\be^i\otimes\bP^j)\,,
\end{align*}
which, recalling \eqref{generators} and \eqref{div_lij_A}, completes the proof.
\end{proof}
  \begin{lemma}[Equivalent formula]\label{th:equiv_torque} For every torque $m_{K}$, $1\leq K\leq n(n-1)/2$, we have
\begin{align*}
       m_{K}[\bsigma,\tilde{M}]&=\int_{\tilde{M}}\mathbf{l}_{K}\odot\DivM\bar{\bsigma}-  \int_{\tilde{M}}\bm{\omega}_{K}\odot\bar{\bsigma}\,,
\end{align*} 
where $\mathbf{l}_K$ and
$\bm{\omega}_{K}$ are given in \eqref{generators} and Lemma~\ref{lemma:l_ij_A}.
\end{lemma}
\begin{proof} The torque terms in Definition~\ref{def:force-torque} can be rearranged via Definition~\ref{def:transpose}  as follows,
\begin{align}\mathbf{l}_{K}\odot \bsigma(\bv)= \mathbf{l}_{K}:\bsigma(\bv)=\mathbf{l}_{K}:(\bar{\bsigma}\cdot\bv)=(\mathbf{l}_{K}:\bar{\bsigma})\cdot\bv\,,
\end{align}
where at the last equality we utilized  Proposition~\ref{associativity}. Application of Lemma~\ref{lemma:l_ij_A} with $\bA=\mathbf{l}_{K}:\bar\bsigma$ yields the claim.
\end{proof}

\begin{theorem}\label{th:normal-at-tangential} If a submanifold $M\subset \mathbb{R}^n$ is in equilibrium under stress $\bsigma\in \mathbb{T}_2(M)$, then pointwise,
\begin{align*}
\bN:\bsigma(\bP\bv)&=0\,.
\end{align*}  
where $\bN$ and $\bP$ are given in \eqref{projection-row-represent}. In other words, the stress vector from Definition~\ref{def:Cauchy_stress} must be tangential, $\bbf_\bv \in \mathbb{PT}_1(M)$, for every tangential orientation $\bv\in  \mathbb{PT}_1(M)$ at any point of $M$.
\end{theorem}
\begin{proof}Consider an arbitrary $\tilde{M}\subset M$ in equilibrium as in Definition~\ref{def:force-torque}. Proposition~\ref{div_sigma_equilib} implies that $\DivM\bar{\bsigma}=0$ pointwise and, consequently, Lemma~\ref{th:equiv_torque} gives
\begin{align*}
       \bm{\omega}_{K}\odot\bar{\bsigma}=0\,,\quad 1\leq K \leq n(n-1)/2\,.
\end{align*}
By recalling the indeces in \eqref{omegas}, one deduces that $(\be^i \otimes \bP^j  )\odot\bar{\bsigma}=(\be^j\otimes \bP^i)\odot\bar{\bsigma}$ and
\begin{align}\label{P_i_sigma_j}
   \bP^i\odot (\bar\bsigma)^j =\bP^j\odot (\bar\bsigma)^i\,,\qquad 1\leq i < j \leq n\,,
\end{align}
since $(\be^i \otimes \bP^j  )\odot\bar{\bsigma}=\textstyle{\sum_*}(\be^i \otimes \bP^j  )^*\odot (\bar\bsigma)^* =\bP^j\odot (\bar\bsigma)^i$ by Definition~\ref{def:left_right_contraction} and Proposition~\ref{outer_represent}. Also, $\bsigma(\bP\bv) =  \bar\bsigma\cdot(\bP\bv)=\{\bar\bsigma^*\cdot\bP\bv\}$  by Definition~\ref{def:transpose} and Proposition~\ref{insertion_rowrep}. Hence, \eqref{P_i_sigma_j} implies
\begin{align*}
\bar\bsigma^*\cdot\bP\bv= (\bP:\bar\bsigma^*)\cdot\bv=
 (\bar\bsigma:\bP^*)\cdot\bv\,,
\end{align*}
and, consequently, the vector $\bsigma(\bP\bv)$ is tangential since, for every normal $\bn^T_k$, $1\leq k \leq m$,
\begin{align*}
\bn_k^T:\bsigma(\bP\bv)=\bn_k^T:\{(\bar\bsigma:\bP^*)\cdot\bv\}
 = (\bar\bsigma:(\textstyle{\sum_*}\bn_k^*:\bP^*))\cdot\bv=  (\bar\bsigma:\bP(\bn_k))\cdot\bv  \,,
\end{align*}
and $\bP(\bn_k)=0$. Recalling \eqref{projection-row-represent} and Definition~\ref{def:left_right_contraction} yield the claim.
\end{proof}
\begin{remark}
    The cancelation of the normal-at-tangential component of the stress vector from Theorem~\ref{th:normal-at-tangential} was also derived in the proof of \cite[Theorem 5.2]{gurtin1975continuum} for elastic surfaces, i.e. submanifolds of codimension one.  Since the authors restrict the consideration to the tangential orientation of the cut, see \cite[(5.10)]{gurtin1975continuum}, they were able to deduce the tangentiality and symmetry below \cite[(5.19)]{gurtin1975continuum}. Without this restriction, Newton laws with the conservation of momentum and angular momentum imply Theorem~\ref{th:normal-at-tangential} only. Consequently, mathematical models involving a stress tensor on embedded submanifolds may very well be non-tangential and, more surprisingly, non-symmetric.
\end{remark}

\section{Application: Dirichlet energy rate for an evolving submanifold}\label{sec:evolving_dirichlet}

We demonstrate the properties of the recursive tensor calculus proposed in Section~\ref{sec:calculus} by working on a submanifold $M(t)$, $t\in (0,T)$ which smoothly evolves through a fixed domain $\Omega\subset\mathbb{R}^n$. We study the following domain-dependent, gradient energy (in the sense of Definition~\ref{nablaM}) for a general rank tensor field $\bT(t)\equiv \bT \in \mathbb{T}_q(M(t))$ that is assigned on every $M(t)$ and for all $t\in (0,T)$.
\begin{definition}[tensorial Dirichlet energy]\label{dirichlet_energy} Given $\bT(t)\in\mathbb{T}_q(M(t))$, $t\in(0,T)$, define
\begin{align*}
E(t)=\frac12\int_{M(t)}|\nablaM \bT|^2\,,\qquad \qquad |\nablaM\bT|^2:=\nablaM \bT \odot \nablaM\bT\,.
\end{align*}
\end{definition}
Motivated by \cite{dziuk_finite_2013}, we will derive the rate of change $\frac{d}{dt}E(t)$ for the general rank, dimension, and codimension. The main instrument is the so-called Reynolds transport theorem, the scalar version of which can be found in \cite{ambrosio2000functions, delfour2011shapes} while the  recursive Definition~\ref{def:integral} easily yields the general rank case as follows. 

\begin{proposition}[Reynolds transport formula]\label{reynolds}
For any ambient field $\tilde{\bT}(t) \in \mathbb{T}_q(\Omega)$, $t\in (0,T)$, and a submanifold $M\equiv M(t)\subset \Omega$ evolving with the material velocity $\bw^T(t) \in \mathbb{T}_1(\Omega)$, we have
     \begin{align*}
       \frac{d}{dt}\int_{M(t)} \tilde{\bT} = \int_{M(t)} \mathcal{D}_\bw^t \tilde{\bT} + \int_{M(t)} (\divM \bw) \tilde{\bT} \,,
    \end{align*}
    where $\mathcal{D}_\bw^t \tilde{\bT} \in \mathbb{T}_q(M(t))$ is given in \eqref{ambient_material}.
    \end{proposition}
    Note that all integrands on the right-hand side of Reynolds transport formula are  understood as instantaneous tensor fields on $M(t)$ and we seek a formula for $\frac{d}{dt}E(t)$ with the same property. We now state the main result of Section~\ref{sec:evolving_dirichlet}, and its proof is postponed to Section~\ref{prof:dirichlet_rate}.
\begin{theorem}\label{thm:dirichlet_rate} Given $\bT(t) \in \mathbb{T}_q(M(t))$ on an evolving submanifold $M(t)$, $t\in (0,T)$, satisfying Assumption~\ref{assump:normal} with the material velocity $\bw^T(t) \in \mathbb{T}_1(M(t))$, the rate of change of the Dirichlet energy  from Definition~\ref{dirichlet_energy} can be expressed  instantaneously on $M(t)$ as follows,
    \begin{align*}
        \frac{dE(t)}{dt}
        =\int_{M(t)} \nablaM \bT \odot \nablaM(\mathcal{D}^t_\bw  \bT)+ \frac12\int_{M(t)}|\nablaM \bT|^2 \divM \bw -\int_{M(t)}(\nablaM \bT \bigcirc\nablaM^\mathrm{cov}\bw^T)\odot{}\nablaM \bT\,, 
    \end{align*}
    where the operator $\mathcal{D}_\bw^t$ and binary operation $\bigcirc$ are given in Definition~\ref{def:material_deriv} and  \eqref{bigcurc}.
\end{theorem}

    \subsection{Material perspective extrinsically}\label{sec:material_perspect}
Before Section~\ref{sec:evolving_dirichlet}, a fixed submanifold with level-set functions $d_i(\bx)$, $1\leq i \leq m$, was considered. To work on a family of submanifolds $M(s)$, we may consider  level-set functions $d_i(\bx,s)$ that depend on a parameter  $s\in [0,S)$. Such level-set functions do not say anything about the evolution of local patches on $M(s)$ -- the type of information that the intrinsic perspective is equipped with by default because coordinate maps are typically assumed to be \textit{material}. In a broader sense, the variable $s$ here parametrizes only the family of shapes $M(s)$, e.g. the actual submanifold evolution may not be uniform in $s$.

To this end, we now enrich the geometric evolution with the  material perspective by considering
a family ${\Pi}$ of non-intersecting \textit{material paths}, 
and a common parametrizer $t\in (0,T)$ of all paths from ${\Pi}$, called \textit{global time}. 
Within  this extrinsic construction, we say that 
\begin{align}\label{evolving_submanifold}
M(t)=\{\bx \in \mathbb{R}^n : \quad \exists ! \gamma \in \Pi \,,\qquad \bx = \gamma(t)\} 
\end{align}
 is an \textit{evolving submanifold} well-defined by a family of material paths $\Pi$ and global time $t$ if, for every $t\in (0,T)$, the projectors in \eqref{projectors} are properly defined for every $\bx \in M(t)$ and the derivatives in \eqref{basics} - for scalar and vector fields on $M(t)$.
For an ambient fixed tensor field $\tilde{\bT}\in{}\mathcal{\bT}_q(\Omega)$, one may consider the \textit{traces} $\bT\in \mathbb{T}_q(M(t))$ and $\hat{\bT}\in \mathbb{T}_q(\gamma)$, for any $t\in (0,T)$ and any $\gamma\in \Pi$, respectively. The validity of these informal arguments is delicate and  depends on the geometric regularity of $M(t)$ and the analytic regularity of the tensor fields under consideration; see \cite{bonito2020finite} for the codimension one case. The following assumption alleviates this complex matter.
\begin{assump}[submanifold and material     traces]\label{assump:evolving}
Assume $M(t)\subset \Omega$, $t\in(0,T)$, is an evolving submanifold well-defined in the sense of \eqref{evolving_submanifold} by a family of material paths $\Pi$ equipped with the global time $t$. Also, for any ambient $\tilde{\bT}\in \mathbb{T}_q(\Omega)$, there exists for every $t\in(0,T)$ an instantaneous $\bT\in \mathbb{T}_q(M(t))$ and $\hat{\bT}\in \mathbb{T}_q(\gamma)$ satisfying
    \begin{align*}
    \bT[\bx] =
    \hat{\bT}[\bx] = \tilde{\bT}[\bx] \,,\qquad \qquad \forall \bx=\gamma(t)\in M(t) \,.
    \end{align*}
\end{assump}
Recall that  $\gamma\in\Pi$ is an  embedded submanifold in the sense of Section~\ref{sec:FTC} so the tensor calculus of  Section~\ref{sec:calculus} is applicable for $\hat{\bT}\in \mathbb{T}_q(\gamma)$; for example,  for every $\gamma\in {\Pi}$, Definition~\ref{def:cartesian} provides   the directional derivative $\mathcal{D}_{\bw}^{\gamma}\hat{\bT} \in \mathbb{T}_q(\gamma)$ with the \textit{material} velocity, $\bw=\frac{d}{dt}\gamma(t)$, which is tangential, $\bw^T\in \mathbb{PT}_1(\gamma)$. Based on Assumption~\ref{assump:evolving}, one may define the ambient material derivative $\mathcal{D}_\bw^t\tilde\bT(t)\in \mathbb{T}_q(M(t))$ as follows,
 \begin{align}\label{ambient_material}
    \mathcal{D}^t_\bw \tilde{\bT}(t) [\bx]:= \pa_t\tilde{\bT}(t)[\bx]+\mathcal{D}_{\bw}^\gamma\hat{\bT}(t)[\bx]\,, \qquad \forall \bx=\gamma(t)\in M(t)\,.
\end{align}

Alternatively to an ambient $\tilde{\bT}\in\mathbb{T}_q(\Omega)$, one may wish to work with a family of tensor fields $\bT(t)\in\mathbb{T}_q(M(t))$, $t\in (0,T)$, and define the  material derivative using  the family of material paths and the global time as follows.

\begin{definition}[material derivative]\label{def:material_deriv} Given a family $\Pi$ of material paths equipped with global time $t$,   $\mathcal{D}^t_\bw\bT\in\mathbb{T}_q(M(t))$ is defined for a time-dependent $\bT(t)\in \mathbb{T}_q(M(t))$\,, $t\in (0,T)$ as follows,
     \begin{align*}
    \mathcal{D}^t_\bw \bT[\bx]:= \lim_{s\rightarrow{}0+}s^{-1}(\bT(t+s)[\gamma(t+s)]-\bT(t)[\bx])\,, \qquad \forall \bx = \gamma(t) \in M(t)\,.
\end{align*}
\end{definition}
The proof of Theorem~\ref{thm:dirichlet_rate} will utilize the following normal extension from a submanifold to an ambient domain that is well-known for codimension one case, see, e.g., \cite{olshanskii_eulerian_2014}, \cite{bonito2020finite}.
\begin{assump}[normal extension]\label{assump:normal}
In addition to Assumption~\ref{assump:evolving}, for every $t\in (0,T)$ and any ${f}\in \mathbb{T}_0(M(t))$, there exists an instantaneous $f^e(t)\in \mathbb{T}_0(\Omega_\delta(t))$, $M(t)\subset\Omega_\delta(t)\subset \Omega$, satisfying
    \begin{align*}
    f^e(t)[\bx] =
    {f}[\bx]  \,,\quad \nabla f^e(t)[\bx] = \nablaM f[\bx]\,,\qquad \qquad \forall \bx\in M(t)\,. 
    \end{align*}
\end{assump}
Note that it follows from Assumption~\ref{assump:normal} and the recursivity of Definitions~\ref{def:row-represent} and \ref{nablaM}  that any tensor field $\bT\in\mathbb{T}_q(M(t))$ can be extended to an $\bT^e(t)\in\mathbb{T}_q(\Omega_\delta(t))$
\begin{align}\label{extension}
    \bT^e(t)[\bx] =
    {\bT}[\bx]  \,,\quad \nabla \bT^e(t)[\bx] = \nablaM \bT[\bx]\,,\qquad \qquad \forall \bx\in M(t)\,. 
    \end{align}
     Having $\bT^e(t)$ instantaneously, one may compute $\mathcal{D}_\bw^t\bT$ as $\pa_t{\bT^e}+\mathcal{D}_{\bw}^\gamma\hat\bT{}^e$  as shown below.
\begin{proposition}\label{equiv_ambient} The material derivative from Definition~\ref{def:material_deriv} of  $\bT(t)\in \mathbb{\bT}_q(M(t))$ can be expressed via ambient material derivative \eqref{ambient_material} using any instantaneous extension $\tilde{\bT}(t)\in \mathbb{T}_q(\Omega_\delta(t))$, i.e.
    \begin{align*}
        \mathcal{D}_\bw^t\bT[\bx] = \mathcal{D}_\bw^t\tilde{\bT}[\bx]\,, \qquad \forall \bx  \in M(t)\,.
    \end{align*}
\end{proposition}
\begin{proof}
Note that the left- and the right-hand sides can be understood as restrictions to each $\gamma\in\Pi$. Their equality for scalars, $q=0$, is a well-known \cite{dziuk2013l2}. Definition~\ref{def:material_deriv} and \eqref{ambient_material} are recursive in $q$ yielding the claim.
\end{proof}

\subsection{Commutators of submanifold, material and ambient derivatives}
We now proceed by computing the material derivative of the projector $\bP$ generalizing \cite{jankuhn_incompressible_2018}.

\begin{proposition}[commutator of the material derivative and the Cartesian gradient]\label{p:commutator} Consider $M(t)$, $t\in (0,T)$, evolving with the material velocity $\bw^T\equiv{}\bw^T(t)\in\mathbb{T}_1(\Omega)$. For any $\bT\equiv\bT(t)\in\mathbb{T}_q(\Omega)$, $q\geq 0$, we have on $M(t)$ that
    \begin{align}
        \nabla{}(\mathcal{D}^t_\bw \bT)-\mathcal{D}^t_\bw{(\nabla{}\bT)}  = 
        \nabla \bT \bigcirc  \nabla\bw^T\,.
    \end{align}
    with  the operation  $\bT \bigcirc \bS \in \mathbb{T}_{q+s-2}$ given recursively for a $\bT\in \mathbb{T}_q$, $q\geq 1$, and $\bS\in \mathbb{T}_s$, $s\geq 1$, by
    \begin{align}\label{bigcurc}
        \bT \bigcirc \bS = \{\bT^*\bigcirc\bS\}\,,\qquad  \bu^T \bigcirc \bS =  \bu^T:\bS\,.
    \end{align}
\end{proposition}
\begin{proof}  Base of induction for $\nabla f\in \mathbb{T}_1(\Omega)$. Proposition~\ref{product_rules} for $\Omega$ with $\bu^T\equiv \nabla f$ and $\bv^T\equiv\bw^T$ gives 
\begin{align*}
    \nabla(\nabla f\odot\bw^T)=\bw^T:(\nabla(\nabla f))+\nabla f\bigcirc{}\nabla\bw^T\,,
\end{align*}
and due to symmetry  $\bw^T:(\nabla(\nabla f))=\nabla(\nabla{f})\cdot\bw$ and independence of $\pa_t$ and $\nabla$ operators,
    \begin{align}
   \pa_t(\nabla f) +\nabla{(\nabla{}f)}\cdot\bw=\nabla(\pa_t f+\nabla f \odot \bw^T)-\nabla{}f:\nabla \bw^T\,,
    \end{align}
    thus completing the base. The induction step follows similarly as  Definition~\ref{def:cartesian} is recursive,
\begin{align}
        \mathcal{D}^t_\bw{(\nabla{}\bT)}= \{\mathcal{D}^t_\bw{(\nabla{}(\bT^*))}\}=\{\nabla{}(\mathcal{D}^t_\bw (\bT^*))\}
        -\{\nabla \bT^*\bigcirc \nabla\bw^T\}\,,
    \end{align}
    where the last term justifies the definition of the  binary operation in \eqref{bigcurc}.
\end{proof}

\begin{remark}
    The tensor operation $\bT\bigcirc\bS \in \mathbb{T}_2$ for rank two tensors $\bT,\bS\in \mathbb{T}_2$  corresponds to the matrix-matrix product $TS$ of matrices $T$ and $S$ representing them. Using the recursive tree analogy of Remark~\ref{tree_analogy} for general rank tensors, one says that the operation $\bigcirc$ right-contracts with $\bS$ every rank one leaf  $\bu^T$  of $\bT$ and replaces it with the tree of the result, $\bu^T:\bS$.
\end{remark}
\begin{proposition}[further properties of $\bigcirc$]\label{nablaMtonabla}
    For $\bT\in\mathbb{T}_q(\Omega)$ and any $M\subset \Omega$, we have with the operation \eqref{bigcurc}  that
    \begin{align*}
        \nablaM\bT = \nabla \bT \bigcirc \bP\,,\qquad \forall\bx\in M\,.
    \end{align*}
    Additionally, the operation \eqref{bigcurc} is associative, i.e. for any $\bT\in \mathbb{T}_q$, $\bS\in \mathbb{T}_s$,  $\bR\in \mathbb{T}_r$, it holds that
    \begin{align*}
    (\bT\bigcirc\bS)\bigcirc \bR =  \bT\bigcirc(\bS\bigcirc \bR) 
\end{align*}
for every $q,s,r\geq 1$ such that both sides are well-defined. 
\end{proposition}
\begin{proof}
    The base of induction, $\nablaM f =\nabla f:\bP=\nabla f \bigcirc\bP$, follows from \eqref{basics} and \eqref{bigcurc}. For $q\geq 1$,
    \begin{align*}
      \nablaM\bT = \{\nablaM\bT^*\}=  \{\nabla \bT^* \bigcirc \bP \}= \nabla \bT \bigcirc \bP\,
    \end{align*}
    where we used the recursivity of Definition~\ref{nablaM}, the induction assumption, and  \eqref{bigcurc}.

    To  show associativity, \eqref{bigcurc} implies that $(\bT\bigcirc\bS)\bigcirc \bR \in \mathbb{T}_{q+s+r-4}$ is equal to
\begin{align*}
    \{((\bT\bigcirc\bS)\bigcirc \bR)^*\} = \{(\bT\bigcirc\bR)^*\bigcirc \bR\}  = \{(\bT^*\bigcirc\bS)\bigcirc \bR\} \,,
\end{align*}
which, upon recursion down to the leafs $\bu^T\in \mathbb{T}_1$ of $\bT$, gives with Definition~\ref{def:left_right_contraction},
\begin{align*}
(\bu^T\bigcirc\bS)\bigcirc\bR= (\sum_*\bu^*:\bS^*)\bigcirc\bR=\sum_*\bu^*\{\bS^*\bigcirc\bR\} = \bu^T:(\bS\bigcirc\bR)= \bu^T\bigcirc(\bS\bigcirc\bR)\,,
\end{align*}
where $\bS\bigcirc\bR \in \mathbb{T}_{s+r-2}$ is the same for every leaf $\bu^T$. The tensor $\bu^T\bigcirc(\bS\bigcirc \bR)\in{}\mathbb{T}_{s+r-3}$ is exactly what one gets recursively in
$\bT\bigcirc(\bS\bigcirc \bR) = \{\bT^*\bigcirc(\bS\bigcirc \bR)\} 
$.
\end{proof}

The next lemma demonstrates that the operators in \eqref{bigcurc} and Definition~\ref{nablaM} interact as leaf children of $\nablaM \bT$  incorporate their  tangency into a rank two tensor $\bA$.
\begin{lemma}\label{vanishing_Cw} For any $\bT, \bS\in \mathbb{T}_q(M)$, $q\geq 0$, and any $\bA \in \mathbb{T}_2(M)$ we have
    \begin{align*}
        (\nablaM \bT \bigcirc \bA) \odot \nablaM\bS  = (\nablaM \bT \bigcirc \mathbb{P}(\bA)) \odot \nablaM\bS\,.
    \end{align*}
\end{lemma}
\begin{proof} According to Definition~\ref{def:tangency} and \eqref{vector-covector_insert}, $\bA(\bP\bu,\bP\bv)=\mathbb{P}(\bA)(\bu,\bv)$ and $\bA(\bu,\bv)=(\bu^T:\bA)(\bv)$ hold for any $\bu, \bv \in \mathbb{R}^n$. Since $\nablaM f, \nablaM g \in \mathbb{PT}_1(M)$ one obtains the base of induction  from \eqref{bigcurc},
\begin{align*}
        (\nablaM f \bigcirc \bA) \odot \nablaM g  =  (\nablaM f : \bA)\odot  \nablaM g = (\nablaM f : \mathbb{P}(\bA))\odot  \nablaM g\,.
\end{align*}
The induction step can be shown using  the recursive Definitions~\ref{def:inner_product}, \ref{nablaM} and \eqref{bigcurc} as follows,
\begin{align*}
        (\nablaM \bT \bigcirc \bA)\odot \nablaM\bS  &= \textstyle{\sum_*} (\nablaM \bT \bigcirc \bA)^*\odot (\nablaM\bS)^* = \textstyle{\sum_*}(\nablaM (\bT^*) \bigcirc \bA)\odot  \nablaM(\bS^*)  \\
        &= \textstyle{\sum_*}(\nablaM (\bT^*)\bigcirc \mathbb{P}(\bA))\odot  \nablaM(\bS^*) =(\nablaM \bT \bigcirc \mathbb{P}(\bA))\odot \nablaM\bS\,,
        \end{align*}
        where the induction assumption is applied in the second line.
\end{proof}

\begin{lemma}\label{p:material_n_P}Assume $|d_i(\bx,t)|=1$, for all $1\leq i\leq m$. Let $M(t)$ evolve with material $\bw$, then
    \begin{align*}
    \mathcal{D}^t_\bw{\bn}_i^T &=
    -\bn_i^T : \nabla\bw^T\,,
        \\
        \mathcal{D}^t_\bw{\bP}&=-\sum_{i=1}^m (\mathcal{D}^t_\bw{\bn_i}\otimes\bn_i+\bn_i\otimes\mathcal{D}^t_\bw{\bn_i})
      =:-2\bC[\bw] \,.
    \end{align*}
    Moreover, $\bC[\bw]\in \mathbb{T}_2(M)$ has a vanishing tangential component, i.e. $\mathbb{P}(\bC[\bw])=0$.
\end{lemma}
\begin{proof} 
Since  $\bn_i = \nabla d_i$ and $d_i=d_i(x,y,...,z,t)$ vanishes on $M(t)$, $\mathcal{D}^t_\bw d_i=0$, and by Corollary~\ref{p:commutator}, 
 \begin{align}
        \mathcal{D}^t_\bw{\bn_i^T}=-\bn_i^T : \nabla\bw^T\,.
    \end{align}
    Since $\bP=\bI-\bN$ and  
$\mathcal{D}^t_\bw{(\bn_i^T\otimes\bn_i^T)}=\mathcal{D}^t_\bw{\bn_i^T}\otimes\bn_i^T+\bn_i^T\otimes\mathcal{D}^t_\bw{\bn_i^T}$  by Proposition~\ref{product_rules}, one obtains $2\bC[\bw]=\mathcal{D}^t_\bw \bN =-\mathcal{D}^t_\bw{\bP}$. Similarly, $\bn_i^T\odot\mathcal{D}^t_\bw{\bn_i^T}=\mathcal{D}^t_\bw{(\bn_i^T\odot \bn^T_i)}=\mathcal{D}^t_\bw{1}=0$. Note that  for $m=1$ this implies $\mathbb{P}(\mathcal{D}^t_\bw\bn_i^T)=0$; however, for a general codimension $m>1$ there could be other $\bn^T_j$, $j\neq i$ such that $\bn_j^T\odot \mathcal{D}^t_\bw\bn_i^T \neq 0$.

By Definition~\ref{def:tangency}, 
$\mathbb{P}(\bC[\bw])(\bu,\bv)=\bC[\bw](\bP\bu,\bP\bv)$ for any $\bu,\bv\in\mathbb{R}^n$. According to  Definition~\ref{outer_represent}, we have for every $1\leq i \leq m$ that
\begin{align*}
(\bn_i^T\otimes\mathcal{D}^t_\bw{\bn_i^T}) (\bP\bu,\bP\bv) = \left(\bn_i^T (\bP\bu)\right) \left(\mathcal{D}^t_\bw{\bn_i^T}(\bP\bv)\right) = 0\,,
\end{align*}
 since $\bP\bu\cdot \bn_i =0$. Proceeding similarly for $\mathcal{D}^t_\bw{\bn_i^T}\otimes{\bn_i^T}$ and $\mathbb{P}(\bC[\bw])=0$ as claimed. 
\end{proof}
\begin{lemma}[commutator of the material derivative and  the submanifold gradient]\label{l:surface_sonnet_virga} Consider $M(t)\subset \Omega$ evolving with the material velocity $\bw^T\equiv{}\bw^T(t)\in\mathbb{T}_1(\Omega)$. For any $\bT=\bT(t)\in\mathbb{T}_q(\Omega)$, $q\geq 0$, we have with \eqref{bigcurc} that
  \begin{align}  \nablaM(\mathcal{D}^t_\bw{\bT})-\mathcal{D}^t_\bw{(\nablaM \bT)}=\nabla \bT\bigcirc(2\textbf{}\bC[\bw]
  +\nablaM\bw^T)\,.
  \end{align}
\end{lemma}
\begin{proof} By Proposition~\ref{nablaMtonabla}, $\nablaM \bT=\nabla{}\bT\bigcirc \bP$ and the product rules give
    \begin{align*}
\mathcal{D}^t_\bw{\nablaM \bT}=\mathcal{D}^t_\bw{(\nabla{}\bT\bigcirc \bP)}=(\nabla{}\bT)\bigcirc(\mathcal{D}^t_\bw{\bP})+\mathcal{D}^t_\bw{(\nabla{}\bT)}\bigcirc\bP\,,
    \end{align*}
 where
$(\nabla{}\bT)\bigcirc\mathcal{D}^t_\bw{\bP}=-(\nabla \bT)\bigcirc2\bC(\bw)
$ by Proposition~\ref{p:material_n_P}. We treat  the second term with Proposition~\ref{p:commutator} and apply $\bigcirc \bP$ to the result,
\begin{align*}
\mathcal{D}^t_\bw{(\nabla{}\bT)}\bigcirc\bP=\nabla(\mathcal{D}^t_\bw{\bT})\bigcirc\bP
-(\nabla{}\bT\bigcirc\nabla\bw^T)\bigcirc\bP\,,
\end{align*} 
where $(\nabla{}\bT\bigcirc\nabla\bw^T)\bigcirc\bP = \nabla{}\bT\bigcirc(\nabla\bw^T\bigcirc\bP) = \nabla{}\bT\bigcirc\nablaM\bw^T$ by Proposition~\ref{nablaMtonabla}.
\end{proof}

\subsection{Proof of Theorem \ref{thm:dirichlet_rate}} \label{prof:dirichlet_rate}
\begin{proof} For each fixed $t\in (0,T)$, one can locally extend a given $\bT(t)\in\mathbb{T}_q(M(t))$  to an ambient field $\bT^e\equiv\bT^e(t)\in \mathbb{T}_q(\Omega_\delta(t))$ satisfying  \eqref{extension} granted by Assumption~\ref{assump:normal}. Hence, for every $t\in (0,T)$, the Dirichlet energy can be equivalently computed as follows,
\begin{align}
   E(t)=\frac12 \int_{M(t)} \nablaM\bT(t) \odot \nablaM\bT(t) =  \frac12\int_{M(t)} \nabla\bT^e(t) \odot \nabla\bT^e(t) \,.
\end{align}
Proposition~\ref{reynolds} and straightforward  product rules give for the rate $\frac{d}{dt}E(t)$ that     \begin{align*}
        \frac{d}{dt}E(t) - \frac12\int_{M(t)}|\nabla \bT^e(t)|^2 \divM\bw = \int _{M(t)} \mathcal{D}^t_\bw(\nabla \bT^e(t)) \odot \nabla
       \bT^e(t)  = \int _{M(t)} \mathcal{D}^t_\bw(\nabla \bT^e(t)) \odot \nablaM
       \bT(t)
       \end{align*}
       and, recalling \eqref{extension} and applying    Lemma~\ref{l:surface_sonnet_virga} with $\bT^e(t)\in \mathbb{T}_q(\Omega_\delta(t))$, we rewrite the last term as
       \begin{align*}
        \int _{M(t)} \mathcal{D}^t_\bw(\nabla \bT^e) \odot \nablaM
       \bT = \int _{M(t)}  \nablaM(\mathcal{D}^t_\bw{\bT^e(t)})\odot\nablaM \bT - \int _{M(t)}(\nabla{}\bT^e(t)\bigcirc (2\bC[\bw]+\nablaM\bw^T))\odot{}\nablaM \bT\,.
    \end{align*}
The claim follows from Proposition~\ref{equiv_ambient}, Lemma~\ref{vanishing_Cw}, Definition~\ref{def:covariant_derivative} and   $\mathbb{P}(\bC[\bw])=0$ by Lemma~\ref{p:material_n_P}.
\end{proof}

\section{Conclusions}

 The presented in Section~\ref{sec:calculus}  recursive version of general rank tensor calculus has an algorithmically favorable structure of complete trees that was introduced in Section~\ref{sec:recursive}. This notational framework has been proven to be useful in the derivation of novel statements in the context of general dimension and codimension  for the three applications in Sections~\ref{sec:Euler_eq}, \ref{sec:stress_submanifold} and \ref{sec:evolving_dirichlet}.
We conclude  with examples of future directions that can emanate   directly from the framework developed in this paper:
\begin{enumerate}
    \item Numerical analysis and computations for tensor-valued PDEs on submanifolds.
    \item Variational calculus on evolving submanifolds and continuum modeling.
    \item Shape classification using the eigenspectra of the Laplace-Beltrami operator for general rank tensors.
\end{enumerate}

\printbibliography
\end{document}